\documentclass[11pt, twoside]{article}
\usepackage{amssymb,amsmath,comment}
\catcode`\@=11 \@addtoreset{equation}{section}
\def\thesection{\arabic{section}}

\def\theequation{\thesection.\arabic{equation}}
\catcode`\@=12
\usepackage{colortbl}%
\usepackage[mathscr]{eucal}
\usepackage{epsf}
\usepackage{a4wide}

\newcommand{\ds} {\displaystyle}
\newcommand{\e}{\epsilon}

\newcommand{\de} {\delta}

\newcommand{\Om} {\Omega}

\newcommand{\De} {\Delta}
\newcommand{\la} {\lambda}

\newcommand{\noi} {\noindent}

\newcommand{\mb} {\mathbb}
\newcommand{\mc} {\mathcal}

\newcommand{\q} {\int_{\Om}(\la|u|^q+\delta|v|^q)\mathrm{d}x}
\newcommand{\Q} {\int_{\Om}(\la|u_k|^q+\delta|v_k|^q)\mathrm{d}x}

\usepackage[all]{xy}
\catcode`\@=11
\def\theequation{\@arabic{\c@section}.\@arabic{\c@equation}}
\catcode`\@=12

\newcommand{\QED}{\rule{2mm}{2mm}}

\newtheorem{Theorem}{Theorem}[section]
\newtheorem{Lemma}[Theorem]{Lemma}
\newtheorem{Proposition}[Theorem]{Proposition}
\newtheorem{Corollary}[Theorem]{Corollary}
\newtheorem{Remark}[Theorem]{Remark}
\newtheorem{Definition}[Theorem]{Definition}

\begin{document}

{\vspace{0.01in}}

\title
{ \sc Doubly nonlocal system with Hardy-Littlewood-Sobolev critical nonlinearity}

\author{J. Giacomoni\footnote{Universit\'{e} de Pau et des Pays de l'Adour, CNRS, LMAP UMR 5142, avenue de l'universit\'{e}, 64013 Pau cedex France. email:jacques.giacomoni@univ-pau.fr}, ~~ T. Mukherjee\footnote{Department of Mathematics, Indian Institute of Technology Delhi,
Hauz Khaz, New Delhi-110016, India.
 e-mail: tulimukh@gmail.com}~ and ~K. Sreenadh\footnote{Department of Mathematics, Indian Institute of Technology Delhi,
Hauz Khaz, New Delhi-110016, India.
 e-mail: sreenadh@gmail.com} }

\date{}

\maketitle

\begin{abstract}

\noi This article concerns about the existence and multiplicity of weak solutions for the following nonlinear doubly nonlocal problem with critical nonlinearity in the sense of Hardy-Littlewood-Sobolev inequality
\begin{equation*}
\left\{
\begin{split}
(-\De)^su &= \la |u|^{q-2}u + \left(\int_{\Om}\frac{|v(y)|^{2^*_\mu}}{|x-y|^\mu}~\mathrm{d}y\right) |u|^{2^*_\mu-2}u\; \text{in}\; \Om\\
 (-\De)^sv &= \delta |v|^{q-2}v + \left(\int_{\Om}\frac{|u(y)|^{2^*_\mu}}{|x-y|^\mu}~\mathrm{d}y \right)  |v|^{2^*_\mu-2}v \; \text{in}\; \Om\\
 u &=v=0\; \text{in}\; \mb R^n\setminus\Omega,
\end{split}
\right.
\end{equation*}
where $\Om$ is a smooth bounded domain in $\mb R^n$, $n >2s$, $s \in (0,1)$, $(-\De)^s$ is the well known fractional Laplacian, $\mu \in (0,n)$, $2^*_\mu = \displaystyle\frac{2n-\mu}{n-2s}$ is the upper critical exponent in the Hardy-Littlewood-Sobolev inequality, $1<q<2$ and $\la,\delta >0$ are real parameters. We study the fibering maps corresponding to the functional associated with $(P_{\la,\delta})$ and show that minimization over suitable subsets of Nehari manifold renders the existence of atleast two non trivial solutions of $(P_{\la,\delta})$ for suitable range of $\la$ and $\delta$.
\medskip

\noi \textbf{Key words:} Nonlocal operator, fractional Laplacian, Choquard equation, Hardy-Littlewood-Sobolev critical exponent.

\medskip

\noi \textit{2010 Mathematics Subject Classification:} 35R11, 35R09, 35A15.

\end{abstract}

\section{Introduction}
Let $\Om \subset \mb R^n$ be a bounded domain with smooth boundary $\partial \Om$ (at least $C^2$), $n>2s$ and $s \in (0,1)$. We consider the following nonlinear doubly nonlocal system with  critical nonlinearity: 
\begin{equation*}
(P_{\la,\delta})\left\{
\begin{split}
(-\De)^su &= \la |u|^{q-2}u + \left(\int_{\Om}\frac{|v(y)|^{2^*_\mu}}{|x-y|^\mu}~\mathrm{d}y\right) |u|^{2^*_\mu-2}u\; \text{in}\; \Om\\
 (-\De)^sv &= \delta |v|^{q-2}v + \left(\int_{\Om}\frac{|u(y)|^{2^*_\mu}}{|x-y|^\mu}~\mathrm{d}y \right)  |v|^{2^*_\mu-2}v \; \text{in}\; \Om\\
 u &=v=0\; \text{in}\; \mb R^n\setminus\Omega,
\end{split}
\right.
\end{equation*}
where $\Om$ is a smooth bounded domain in $\mb R^n$, $n >2s$, $s \in (0,1)$, $\mu \in (0,n)$, $2^*_\mu = \displaystyle\frac{2n-\mu}{n-2s}$ is the upper critical exponent in the Hardy-Littlewood-Sobolev inequality, $1<q<2$, $\la,\delta >0$ are real parameters and  $(-\De)^s$ is the fractional Laplace operator defined  as
{$$ (-\De)^s u(x) = 2C^n_s\mathrm{P.V.}\int_{\mb R^n} \frac{u(x)-u(y)}{\vert x-y\vert^{n+2s}}\,\mathrm{d}y$$
where $\mathrm{P.V.}$ denotes the Cauchy principal value and $C^n_s=\pi^{-\frac{n}{2}}2^{2s-1}s\frac{\Gamma(\frac{n+2s}{2})}{\Gamma(1-s)}$, $\Gamma$ being the Gamma function.} The fractional  Laplacian is the infinitesimal generator of {L\'evy} stable diffusion process and arise in anomalous diffusion in plasma, population dynamics, geophysical fluid dynamics, flames propagation, chemical reactions in liquids and American options in finance, see \cite{da} for instance. We also refer \cite{radu2} to readers for  {{a detailed study on variational methods for fractional elliptic problems.}}

 \noi In the local case, authors in \cite{BJS} studied the existence of of ground states for the nonlinear Choquard equation
\begin{equation}\label{cho1}
 -\De u + V(x)u = \left( \int_\Om \frac{|u(y)|^p}{|x-y|^{\mu}}~\mathrm{d}y \right)|u|^{p-2}u \; \text{ in } \mb R^n,
 \end{equation}
where $p>1$ and $n\geq 3$.
 Recently, Ghimenti, Moroz and Schaftingen \cite{GMS}  proved the existence of least action nodal solution for the problem
\[-\De u + u = (I_\alpha * |u|^2)u \; \text{in} \; \mb R^n,\]
where $\ast$ denotes the convolution and $I_\alpha$ denotes the Riesz potential. Further results related to Choquard equations can be found in the survey paper \cite{survey} and the references therein. Alves, Figueiredo and Yang \cite{AFY} proved existence of a nontrivial solution via penalization method for the following Choquard equation
\begin{equation*}
-\De u+V(x)u= (|x|^{-\mu}*F(u))f(u)\; \text{in}\; \mb R^n,
\end{equation*}
 where $0 <\mu < N,\; N = 3, \;V$ is a continuous real { valued function} and $F$ is the primitive of function $f$. In the nonlocal case, Choquard equations involving fractional Laplacian is a recent topic of research. Authors in \cite{avsq} obtained regularity, existence, nonexistence, symmetry
as well as decays properties for the problem
\[(-\De)^s u + \omega u = (|x|^{\alpha-n}\ast |u|^p)|u|^{p-2}u\; \text{in} \; \mb R^n, \]
where $\omega>0$, $p>1$ and $s \in (0,1)$. Fractional Choquard equations also known as  nonlinear fractional Schr\"{o}dinger equations with Hartree-type nonlinearity arise in  the study of mean field limit of weakly interacting molecules, physics of multi particle systems and the quantum mechanical theory, etc. These are recently studied by some  authors in \cite{chenliu,chosquas2,luxu}.\\

\noi { Concerning the boundary value problems involving the Choquard nonlinearity, the Brezis-Nirenberg type problem that is
\begin{equation*}
\begin{split}
-\De u = \la u + \left( \int_\Om \frac{|u|^{2*_\mu}}{|x-y|^\mu}~\mathrm{d}y \right)|u|^{2^*_\mu-2}u\; \text{in}\; \Om,\; u = 0 \;\text{in}\; \mb R^n \setminus \Om
\end{split}
\end{equation*}
where $\Om$ is bounded domain in $\mb R^n$, was studied by Gao and Yang in \cite{gao-yang}. They proved the existence, multiplicity and nonexistence results for a range of $\la$. Moreover, in \cite{gao-yang2} authors proved the existence results for a class of critical Choquard equations in critical case. Among the very recent works, we cite \cite{non-homo} where Shen, Gao and Yang obtained existence of multiple solutions for non-homogenous critical Choquard equation using the variational methods when $0<\la<\la_1$, where $\la_1$ denotes the first eigenvalue of $-\De$ with Dirichlet boundary condition. }\\

 \noi Coming to the system of equations, elliptic systems involving fractional Laplacian and homogeneous nonlinearity has been studied in \cite{GPS,Sqsn, Fan} using Nehari manifold techniques. Guo et al. in \cite{guo} studied  a nonlocal system involving fractional Sobolev critical exponent and fractional Laplacian.  We also cite \cite{choi,faria,ww} as some very recent works on the study of fractional elliptic systems.
 However there is not much literature available on fractional elliptic system involving Choquard type nonlinearity.
And fractional elliptic system with critical Choquard inequality has not been studied yet, to the best of our knowledge. \\

\noi In this present paper, we discuss the existence and multiplicity result for the problem $(P_{\la,\delta})$. We seek help of the Nehari manifold techniques where minimization over suitable components  of Nehari manifold provide the weak solution to the problem. We divide the problem into two cases that is $0<\mu\leq 4s$ and $\mu > 4s$ and show existence of atleast two solution while bounding the parameters $\la$ and $\delta$ optimally. The existence results in the first case is optimal in the sense of obtaining the constant $\Theta$ (defined in Lemma \ref{Theta-def-lem}). We also reach the expected first critical level that is 
\[I_{\la,\delta}(u_1,v_1)+  \frac{n-\mu+2s}{2n-\mu} \left(\frac{C_s^n \tilde S_s^H}{2} \right)^{\frac{2n-\mu}{n-\mu+2s}}\]
 where $(u_1,v_1)$ denote the first solution of $(P_{\la,\delta})$, in this case (see Lemma \ref{l^-<PScrit-lev}) analogously to the local setting case (refer Lemma $2.4$ in \cite{gao-yang2}). Whereas in the latter case, we obtain the multiplicity for a smaller range of $\la$ and $\delta$ that is $\Theta_0$ (defined in Theorem \ref{sec-sol}) as compared to $\Theta$. We use the blow up analysis involving the minimizers of the embeddings to achieve the goal. In the case $0<\mu\leq 4s$, our results are sharp in the sense that the restrictions on the parameters $\la$ and $\de$ are used only to show that Nehari set is a manifold. Moreover using an iterative scheme, regularity results known for nonlocal problems involving fractional laplacian and strong maximum principle, we show the existence of a positive solution (see Proposition \ref{positive-sol}).

{\begin{Theorem}\label{MT}
Assume $1<q<2$ and $0<\mu<n$ then there exists a positive constants $\Theta$ and $\Theta_0$ such that
\begin{enumerate}
\item if $\mu\leq 4s$ and $ 0< \la^{\frac{2}{2-q}}+ \delta^{\frac{2}{2-q}}< \Theta$, the system $(P_{\la,\delta})$ admits at least two nontrivial solutions,
\item if $\mu> 4s$ and $ 0< \la^{\frac{2}{2-q}}+ \delta^{\frac{2}{2-q}}< \Theta_0$, the system $(P_{\la,\delta})$ admits at least two nontrivial solutions.
\end{enumerate}
Moreover, there exists a positive solution for $(P_{\la,\delta})$.
\end{Theorem}}

\begin{Remark}
We remark that the solution obtained for $(P_{\la,\delta})$ (other than the positive solution) is not even semi trivial. The proof follows along the same line as section $5$(pp. $841$) of \cite{chensquas}.
\end{Remark}

\noi Our paper is organized as follows: Section 2 contains the functional setting and { various asymptotic estimates involving minimizers of  best constants}.  We analyse the fibering maps associated to the Nehari manifold in section 3. Lastly, section 4 contains the proof of main result where we show the existence of atleast two non trivial solutions.

\section{Function Spaces and some asymptotic estimates}
Consider the function space $H^s(\Om)$ as the usual fractional Sobolev space $W^{s,2}(\Om)$ defined by
\[H^s(\Om)= \left\{u \in L^2(\Om): \; \int_{\Om}\int_{\Om}\frac{|u(x)-u(y)|^2}{|x-y|^{n+2s}}\mathrm{d}x\mathrm{d}y < +\infty\right\}.\]
Setting $Q:= \mb R^{2n} \setminus (\mc C \Om \times \mc C \Om)$ where $\mc C \Om = \mb R^n \setminus \Om $, we define the Banach space
\[X:= \left\{u:\mb R^n \to \mb R \text{ measurable }:\; u \in L^2(\Om), \int_{Q}\frac{|u(x)-u(y)|^2}{|x-y|^{n+2s}}\mathrm{d}x\mathrm{d}y < +\infty\right\} \]
with the norm defined as
\[\|u\|_X: = \|u\|_{L^2(\Om)}+ \left(\int_{Q}\frac{|u(x)-u(y)|^2}{|x-y|^{n+2s}}\mathrm{d}x\mathrm{d}y \right)^{\frac12} {= \|u\|_{L^2(\Om)}+ \left(\frac{1}{C^n_s} \int_{\Om}u(-\De)^su~\mathrm{d}x\mathrm{d}y\right)^{\frac12}.} \]
If we set $X_0:= \{u \in X:\; u=0 \;\text{in}\; \mb R^n\setminus \Om\}$, then it can be shown that $X_0$ forms a Hilbert space with the inner product
\[\langle u, v\rangle = \int_{Q}\frac{(u(x)-u(y))(v(x)-v(y))}{|x-y|^{n+2s}}\mathrm{d}x \mathrm{d}y\]
for $u,v \in X_0$ and thus the corresponding norm is
\[ \|u\|_{X_0}=\|u\|:= \left(\int_{Q}\frac{|u(x)-u(y)|^2}{|x-y|^{n+2s}}\mathrm{d}x\mathrm{d}y \right)^{\frac12}.\]
Then $X_0$ can be equivalently considered as completion of $C^\infty_0(\Om)$ under the norm $\|\cdot\|_{X}$. It holds that $X_0 \hookrightarrow L^r(\Om)$
continuously for $r \in [1,2^*_s]$ and compactly for $r \in [1,2^*_s)$, where $2^*_s =\displaystyle\frac{2n}{n-2s}$. Now consider the product space $Y:= X_0\times X_0$ endowed with the norm $\|(u,v)\|^2:=\|u\|^2+\|v\|^2$. Before defining the weak solution for $(P_{\la,\delta})$, we need to certify that whenever $u \in X_0$, the term
\[\int_{\Om}(|x|^{-\mu}\ast|u|^{2^*_\mu})|u|^{2^*_\mu}\mathrm{d}x = \int_\Om \int_\Om \frac{|u(x)|^{2^*_\mu}|v(y)|^{2^*_\mu}}{|x-y|^\mu}~\mathrm{d}x \mathrm{d}y\]
is well defined. This is certified by the following well known Hardy-Littlewood-Sobolev inequality.
 \begin{Proposition}\label{HLS}
(\textbf {Hardy-Littlewood-Sobolev inequality}) {[pp. 106, Theorem 4.3, \cite{lieb}]} Let $t,r>1$ and $0<\mu<n $ with $1/t+\mu/n+1/r=2$, $f \in L^t(\mathbb R^n)$ and $h \in L^r(\mathbb R^n)$. There exists a sharp constant $C(t,n,\mu,r)$, independent of $f,h$ such that
 \begin{equation}\label{HLSineq}
 \int_{\mb R^n}\int_{\mb R^n} \frac{f(x)h(y)}{|x-y|^{\mu}}\mathrm{d}x\mathrm{d}y \leq C(t,n,\mu,r)\|f\|_{L^t(\mb R^n)}\|h\|_{L^r(\mb R^n)}.
 \end{equation}
{ If $t =r = \textstyle\frac{2n}{2n-\mu}$ then
 \[C(t,n,\mu,r)= C(n,\mu)= \pi^{\frac{\mu}{2}} \frac{\Gamma\left(\frac{n}{2}-\frac{\mu}{2}\right)}{\Gamma\left(n-\frac{\mu}{2}\right)} \left\{ \frac{\Gamma\left(\frac{n}{2}\right)}{\Gamma(n)} \right\}^{-1+\frac{\mu}{n}}.  \]
 In this case there is equality in \eqref{HLSineq} if and only if $f\equiv (constant)h$ and
 \[h(x)= A(\gamma^2+ |x-a|^2)^{\frac{-(2n-\mu)}{2}}\]
 for some $A \in \mathbb C$, $0 \neq \gamma \in \mathbb R$ and $a \in \mathbb R^n$.}
 \end{Proposition}
 \begin{Remark}
For $u \in {H^s(\mb R^n)}$, if we let $f = h= |u|^p$ then by Hardy-Littlewood-Sobolev inequality,
\[ \int_{\mb R^n}\int_{\mb R^n} \frac{|u(x)|^p|u(y)|^p}{|x-y|^{\mu}}\mathrm{d}x\mathrm{d}y\]
 is well defined for all $p$ satisfying
 \[2_\mu:=\left(\frac{2n-\mu}{n}\right)\leq p \leq \left(\frac{2n-\mu}{n-2s}\right):= 2^*_\mu.\]
\end{Remark}
\noi Next result is a basic inequality whose proof can be worked out in a similar manner as proof of Proposition 3.2(3.3) of \cite{GhSc}.
\begin{Lemma}\label{ineq}
For $u,v \in L^{\frac{2n}{2n-\mu}}(\mb R^n)$, we have
\[\int_{\mb R^n }\int_{\mb R^n}\frac{|u(x)|^p|v(y)|^p}{|x-y|^\mu}~dxdy \leq \left(\int_{\mb R^n}\int_{ \mb R^n}\frac{|u(x)|^p|u(y)|^p}{|x-y|^\mu}~dxdy \right)^{\frac12}  \left(\int_{\mb R^n}\int_{ \mb R^n}\frac{|v(x)|^p|v(y)|^p}{|x-y|^\mu}~dxdy \right)^{\frac12},\]
where $\mu\in (0,n)$ and $p \in [2_\mu,2^*_\mu]$.
\end{Lemma}
\begin{proof}
{We recall the semigroup property of the Riesz potential which states that if $I_\alpha: \mb R^n \to \mb R$ denotes the Riesz potential given by
\[I_\alpha = \frac{A_\alpha}{|x|^{n-\alpha}},\; \text{where} \; A_\alpha = \frac{\Gamma\left( \frac{n-\alpha}{2}\right)}{\Gamma\left(\frac{\alpha}{2}\right) \pi^{n/2}2^\alpha}.\]
Then $I_\alpha$ satisfies $I_\alpha = I_{\alpha/2}\ast I_{\alpha/2}$. Using this alongwith H\"{o}lder's inequality we obtain
\begin{equation*}
\begin{split}
&\int_{\mb R^n }\int_{\mb R^n}\frac{|u(x)|^p|v(y)|^p}{|x-y|^\mu}~dxdy \\
&= \frac{1}{A_{n-\mu}}\int_{\mb R^n}(I_{n-\mu}\ast |u|^p)|v|^p\mathrm{d}x = \frac{1}{A_{n-\mu}}\int_{\mb R^n}(I_{\frac{n-\mu}{2}}\ast |u|^p)(I_{\frac{n-\mu}{2}} \ast |v|^p)\mathrm{d}x\\
& \leq \frac{1}{A_{n-\mu}}\left(\int_{\mb R^n}(I_{\frac{n-\mu}{2}}\ast |u|^p)^2\mathrm{d}x\right)^{1/2} \left(\int_{\mb R^n}(I_{\frac{n-\mu}{2}} \ast |v|^p)^2\mathrm{d}x\right)^{1/2}\\
& = \left(\int_{\mb R^n}\int_{ \mb R^n}\frac{|u(x)|^p|u(y)|^p}{|x-y|^\mu}~dxdy \right)^{\frac12}  \left(\int_{\mb R^n}\int_{ \mb R^n}\frac{|v(x)|^p|v(y)|^p}{|x-y|^\mu}~dxdy \right)^{\frac12}
\end{split}
\end{equation*}.\hfill{\QED}}
\end{proof}

\noi Therefore, it easily follows using Lemma \ref{ineq} that for every $(u,v)\in Y$,
$\displaystyle \int_{\Om}(|x|^{-\mu}\ast|u|^{2^*_\mu})|v|^{2^*_\mu}\mathrm{d}x < +\infty$.
In the context of Hardy- Littlewood-Sobolev inequality that is Proposition \ref{HLS}, for any $u \in X_0$ we get a constant $C>0$ such that
\begin{equation}\label{ineq1}
\int_{\Om}(|x|^{-\mu}\ast|u|^{2^*_\mu})|u|^{2^*_\mu}\mathrm{d}x=\int_{\Om}\int_{\Om}\frac{|u(x)|^{2^*_\mu}|u(y)|^{2^*_\mu}}{|x-y|^\mu}dxdy \leq  C\|u\|_{L^{2^*_s}(\Om)}^{22^*_\mu}.
\end{equation}
For notational convenience, if $u, v \in X_0$ we set
\[ B(u,v):= \int_{\Om}(|x|^{-\mu}\ast|u|^{2^*_\mu})|v|^{2^*_\mu}. \]
\begin{Definition}
We say that $(u,v)\in Y$ is a weak solution to $(P_{\la,\delta})$ if for every $(\phi,\psi)\in Y$, it satisfies
\begin{equation*}
\begin{split}\label{weak-sol}
{C_s^n}(\langle u, \phi\rangle + \langle v,\psi\rangle) &= \int_{\Om}(\la |u|^{q-2}u\phi+\delta |v|^{q-2}v\psi)\mathrm{d}x\\
&\quad+ \int_{\Om}(|x|^{-\mu}\ast|v|^{2^*_\mu})|u|^{2^*_\mu-2}u\phi~\mathrm{d}x +  \int_{\Om}(|x|^{-\mu}\ast|u|^{2^*_\mu})|v|^{2^*_\mu-2}v\psi~\mathrm{d}x.
 \end{split}
\end{equation*}
\end{Definition}
Equivalently, if we define the functional $I_{\la,\delta}:Y\to \mb R$ as
\[I_{\la,\delta}(u):= \frac{{C_s^n} }{2} \|(u,v)\|^2-\frac{1}{q}\int_{\Om}(\la |u|^q+\delta |v|^q)-\frac{2}{22^*_\mu}B(u,v) \]
then the critical points of $I_{\la,\delta}$ correspond to the weak solutions of $(P_{\la,\delta})$. A direct computation leads to $I_{\la,\delta} \in C^1(Y,\mb R)$ such that for any $(\phi,\psi) \in Y$
\begin{equation}\label{Iprime}
\begin{split}
(I_{\la,\delta}^\prime(u,v), (\phi,\psi))&= {C_s^n}(\langle u, \phi\rangle+ \langle v,\psi\rangle) - \int_{\Om}(\la |u|^{q-2}u\phi+\delta |v|^{q-2}v\psi)~\mathrm{d}x
\\
&\quad- \int_{\Om}(|x|^{-\mu}\ast|v|^{2^*_\mu})|u|^{2^*_\mu-2}u\phi~\mathrm{d}x -  \int_{\Om}(|x|^{-\mu}\ast|u|^{2^*_\mu})|v|^{2^*_\mu-2}v\psi~\mathrm{d}x.
\end{split}
\end{equation}

We define
 \begin{equation*}
 {S_s} = \inf_{u \in X_0\setminus \{0\}} \displaystyle\frac{\displaystyle\int_{\mb R^{2n}} \frac{|u(x)-u(y)|^2}{|x-y|^{n+2s}}{\,\mathrm{d}x\mathrm{d}y}}{\displaystyle\left(\int_{\mb R^n}|u|^{2^*_s}\,\mathrm{d}x\right)^{2/2^*_s}} = \inf_{u \in X_0\setminus \{0\}} \frac{\|u\|^2}{\|u\|_{L^{2^*_s}(\mb R^n)}^2} .
 \end{equation*}
\noi Consider the family of functions $\{U_{\epsilon}\}$ defined as
\begin{equation}\label{minimizers}
 U_{\epsilon}(x) = \epsilon^{-\frac{(n-2s)}{2}}\; u^*\left(\frac{x}{\epsilon}\right),\; x \in \mb R^n
 \end{equation}
where $u^*(x) = \bar{u}\left(\frac{x}{S_s^{\frac{1}{2s}}}\right),\; \bar{u}(x) = \frac{\tilde{u}(x)}{{\|\tilde u \|_{L^{2^*_s}(\mb R^n)}}}$ and $\tilde{u}(x)= \alpha(\beta^2 + |x|^2)^{-\frac{n-2s}{2}}$ with $\alpha \in \mb R \setminus \{0\}$ and $ \beta >0$ are fixed constants. Then for each $\epsilon > 0$, $U_\epsilon$ satisfies
\[ (-\De)^su = |u|^{2^*_s-2}u \; \;\text{in} \; \mb R^n \]
and verifies the equality
\begin{equation}\label{eu-lag}
{\int_{\mb R^n} \int_{\mb R^n} \frac{|U_\epsilon(x)-U_{\epsilon}(y)|^2}{|x-y|^{n+2s}}\,\mathrm{d}x\mathrm{d}y =\int_{\mb R^n} |U_\epsilon|^{2^*_s} \,\mathrm{d}x}= {S_s^{\frac{n}{2s}}}.
\end{equation}
For a proof, we refer to \cite{sv}. Next, in spirit of the inequality \eqref{ineq1} we define the best constant
 \begin{equation*}
 S_s^H := \inf\limits_{u \in X_0\setminus\{0\}} \displaystyle\frac{\displaystyle\int_{\mb R^{2n}} \frac{|u(x)-u(y)|^2}{|x-y|^{n+2s}}{\,\mathrm{d}x\mathrm{d}y}}{\displaystyle\left(\int_{\mb R^n} (|x|^{-\mu}\ast|u|^{2^*_\mu})|u|^{2^*_\mu}\mathrm{d}x\right)^{\frac{1}{2^*_{\mu}}}}= \inf\limits_{u \in X_0\setminus\{0\}} \frac{\|u\|^2}{B(u,u)^{\frac{1}{2^*_\mu} }}.
 \end{equation*}
\begin{Lemma}\label{S-H-minimizers}
{ The constant $S_s^H$ is achieved by $u$ if and only if $u$ is of the form
\[C\left( \frac{t}{t^2+|x-x_0|^2}\right)^{\frac{n-2s}{2}}, \; \; x \in \mb R^n \]
for some $x_0 \in \mb R^n$, $C>0$ and $t>0$. Moreover,
 \begin{equation}\label{relation}
S^H_s = \frac{S_s}{C(n,\mu)^{\frac{1}{{2^*_{\mu}}}}}.
\end{equation}}
 \end{Lemma}
 \begin{proof}
{ By the Hardy-Littlewood-Sobolev inequality we easily get that
 \[S^H_s\geq  \frac{S_s}{C(n,\mu)^{\frac{1}{{2^*_{\mu}}}}}.\]
 Also from Proposition \ref{HLS} we know that the inequality in \eqref{HLSineq} is an equality if and only if $u$ is of the form
\[C\left( \frac{t}{t^2+|x-x_0|^2}\right)^{\frac{n-2s}{2}}, \; \; x \in \mb R^n .\]
While we know that if $u$ is of this form then it also forms a minimizer for the constant $S_s$, thus we obtain the result and \eqref{relation} follows directly.\hfill{\QED}}
 \end{proof}

\noi We set
\[\tilde S_s^H = \inf_{(u,v) \in Y\setminus \{(0,0)\}} \frac{\|(u,v)\|^2 }{ \left( \int_{\Om}(|x|^{-\mu}\ast |u|^{2^*_\mu})|v|^{2^*_\mu}~\mathrm{d}x\right)^{\frac{1}{2^*_\mu}}} = \inf_{(u,v) \in Y\setminus \{(0,0)\}} \frac{\|(u,v)\|^2}{B(u,v)^{\frac{1}{2^*_\mu}}} \]
and show the relation between $S_s^H$ and $\tilde S_s^H$ in the following lemma. The argument follows closely the line of Lemma $3.3$ of \cite{chensquas} but for sake of completeness, we include it here.

\begin{Lemma}\label{reltn}
There holds $\tilde S_s^H = 2 S_s^H$.
\end{Lemma}
\begin{proof}
Let $\{g_k\} \subset X_0$ be a minimizing sequence for $S_s^H$. Let $r_1,r_2>0$ be specified later and set the sequences $u_k = r_1g_k$ and $v_k = r_2g_k$ in $X_0$. From the definition of $S_s^H$ we have
 \begin{equation}\label{reltn1}
 \tilde S_s^H \leq \left( \frac{r_1^2+r_2^2}{r_1r_2}\right) \left( \frac{\|g_k\|^2}{B(g_k,g_k)^{\frac{1}{2^*_\mu}}}\right)= \left( \frac{r_1}{r_2}+ \frac{r_2}{r_1}\right)   \left( \frac{\|g_k\|^2}{B(g_k,g_k)^{\frac{1}{2^*_\mu}}}\right).
 \end{equation}
Let us define the function $f:\mb R^+ \to \mb R^+$ by setting $f(x) = x+x^{-1}$. Then it is easy to see that $f$ attains its minimum at $x_0= 1$ with the minimum value $f(1)=2$. We choose $r_1,r_2$ in \eqref{reltn1} such that ${r_1}= {r_2}$ and letting $k \to \infty$ in \eqref{reltn1} we get
\begin{equation}\label{reltn2}
\tilde S_s^H \leq 2S_s^H.
\end{equation}
To prove the reverse inequality we consider the minimizing sequence $\{(u_k,v_k)\} \subset Y \setminus \{(0,0)\}$ for $\tilde S_s^H$. We set $w_k = r_kv_k$ for $r_k>0$ with $B(u_k,u_k)= B(w_k,w_k)$. This alongwith Lemma \ref{ineq} gives
\[B(u_k,w_k) \leq B(u_k,u_k)^{\frac{1}{2}} B(w_k,w_k)^{\frac{1}{2}} = B(u_k,u_k)= B(w_k,w_k).\]
Thus we obtain
\begin{align*}
\frac{\|(u_k,v_k)\|^2}{B(u_k,v_k)^{\frac{1}{2^*_\mu}}}  = r_k \frac{\|(u_k,v_k)\|^2}{B(u_k,w_k)^{\frac{1}{2^*_\mu}}} &\geq r_k \frac{\|u_k\|^2 }{B(u_k,u_k)^{\frac{1}{2^*_\mu}}} + r_k r_k^{-2} \frac{\|w_k\|^2 }{B(w_k,w_k)^{\frac{1}{2^*_\mu}}}\\
& \geq f(r_k) S_s^H \geq 2S_s^H.
\end{align*}
Now passing on the limit as $k \to \infty$ we get
\begin{equation}\label{reltn3}
2 S_s^H \leq \tilde S_s^H.
\end{equation}
Finally from \eqref{reltn2} and \eqref{reltn3} it follows that $S_s^H = 2 \tilde S_s^H$. \hfill{\QED}
\end{proof}

 We recall the definition of $U_\e$ from \eqref{minimizers}. Without loss of generality, we assume $0 \in \Om$ and fix $\delta>0$ such that $B_{4\delta}\subset \Om$. Let $\eta \in C^{\infty}(\mb R^n)$ be such that $0\leq \eta \leq 1$ in $\mb R^n$, $\eta \equiv 1$ in $B_{\delta}$ and $\eta \equiv 0$ in $\mb R^n \setminus B_{2\delta}$. For $\epsilon >0$, we denote by $u_{\epsilon}$ the following function
\[u_\epsilon(x) = \eta(x)U_{\epsilon}(x),\]
for $x \in \mb R^n$, where $U_\epsilon$ is defined in section 2. We have the following results for $u_\epsilon$ from Proposition $21$ and $22$ of \cite{sv}.
\begin{Proposition}\label{estimates}
Let $s\in(0,1)$ and $n>2s$. Then, the following estimates holds true as $\epsilon \rightarrow 0 $ \\
\begin{enumerate}
\item[(i)] $\ds \int_{\mb R^{2n}}\frac{|u_\epsilon(x)- u_\epsilon(y)|^2}{|x-y|^{n+2s}}~\mathrm{d}x\mathrm{d}y {\leq} S_s^{\frac{n}{2s}}+O(\epsilon^{n-2s})$,\\
\item[(ii)] $\ds \int_{\Om}|u_\epsilon|^{2^*_s}~\mathrm{d}x = S_s^{\frac{n}{2s}}+ O(\epsilon^n)$,\\
\item[(iii)] $$\int_{\Om}|u_{\epsilon}(x)|^2~\mathrm{d}x {\geq}
\left\{
    \begin{array}{ll}
        C_s\epsilon^{2s}+O(\epsilon^{n-2s}) & \mbox{if } n>4s \\
        C_s\epsilon^{2s}|\log \epsilon|+O(\epsilon^{2s}) & \mbox{if } n=4s \\
        C_s\epsilon^{n-2s}+O(\epsilon^{2s}) & \mbox{if } n<4s
    \end{array}\hspace{5.0cm}
\right.,$$
\end{enumerate}
for some positive constant $C_s$, depending on $s$.
\end{Proposition}
\noi Using \eqref{relation}, Proposition \ref{estimates}(i) can be written as
\begin{equation}\label{esti-new}
\int_{\mb R^n}\frac{|u_\epsilon(x)- u_\epsilon(y)|^2}{|x-y|^{n+2s}}~\mathrm{d}x\mathrm{d}y \leq S_s^{\frac{n}{2s}}+O(\epsilon^{n-2s})=\left((C(n,\mu))^{\frac{n-2s}{2n-\mu}}S^H_s\right)^{\frac{n}{2s}}+ O(\epsilon^{n-2s}).
\end{equation}
\begin{Proposition}\label{estimates1}
The following estimates holds true:
\begin{equation*}
\left(\int_{\Om}\int_{\Om}\frac{|u_\epsilon(x)|^{2^*_{\mu}}|u_\epsilon(y)|^{2^*_{\mu}}}{|x-y|^{\mu}}~\mathrm{d}x\mathrm{d}y \right)^{\frac{n-2s}{2n-\mu}}\leq (C(n,\mu))^{\frac{n(n-2s)}{2s(2n-\mu)}} (S^H_s)^{\frac{n-2s}{2s}}+ {O(\epsilon^{n})},
\end{equation*}
and
\begin{equation*}
\left(\int_{\Om}\int_{\Om}\frac{|u_\epsilon(x)|^{2^*_{\mu}}|u_\epsilon(y)|^{2^*_{\mu}}}{|x-y|^{\mu}}~\mathrm{d}x\mathrm{d}y \right)^{\frac{n-2s}{2n-\mu}}\geq {\left((C(n,\mu))^{\frac{n}{2s}} (S^H_s)^{\frac{2n-\mu}{2s}}- O\left(\epsilon^{{2n-\mu}}\right)\right)^{\frac{n-2s}{2n-\mu}}.}
\end{equation*}
\end{Proposition}
\begin{proof}
By Hardy-Littlewood-Sobolev inequality, Proposition \ref{estimates}(ii) and \ref{relation}, we get
\begin{equation*}
\begin{split}
&\left(\int_{\Om}\int_{\Om}\frac{|u_\epsilon(x)|^{2^*_{\mu}}|u_\epsilon(y)|^{2^*_{\mu}}}{|x-y|^{\mu}}~\mathrm{d}x\mathrm{d}y \right)^{\frac{n-2s}{2n-\mu}}\hspace{7.0cm} \\
&\leq (C(n,\mu))^{\frac{n-2s}{2n-\mu}} \|u_\epsilon\|^2_{L^{2^*_s}(\Om)}=(C(n,\mu))^{\frac{n-2s}{2n-\mu}}\left(S_s^{\frac{n}{2s}}+ O(\epsilon^n)\right)^{\frac{n-2s}{n}}\\
&= (C(n,\mu))^{\frac{n-2s}{2n-\mu}}\left( (C(n,\mu))^{\frac{n(n-2s)}{2s(2n-\mu)}}(S^H_s)^{\frac{n}{2s}}+ O(\epsilon^n)\right)^{\frac{n-2s}{n}}\\
&=(C(n,\mu))^{\frac{n(n-2s)}{2s(2n-\mu)}} (S^H_s)^{\frac{n-2s}{2s}}+ {O(\epsilon^{n})}.
\end{split}
\end{equation*}
Next, we consider
\begin{equation}\label{har-esti3}
\begin{split}
&\int_{\Om}\int_{\Om}\frac{|u_\epsilon(x)|^{2^*_{\mu}}|u_\epsilon(y)|^{2^*_{\mu}}}{|x-y|^{\mu}}~\mathrm{d}x\mathrm{d}y \\
 &\geq \int_{B_{\delta}}\int_{B_{\delta}}\frac{|u_\epsilon(x)|^{2^*_{\mu}}|u_\epsilon(y)|^{2^*_{\mu}}}{|x-y|^{\mu}}~\mathrm{d}x\mathrm{d}y
 =\int_{B_{\delta}}\int_{B_{\delta}}\frac{|U_\epsilon(x)|^{2^*_{\mu}}|U_\epsilon(y)|^{2^*_{\mu}}}{|x-y|^{\mu}}~\mathrm{d}x\mathrm{d}y\\
& = \int_{\mb R^n}\int_{\mb R^n}\frac{|U_\epsilon(x)|^{2^*_{\mu}}|U_\epsilon(y)|^{2^*_{\mu}}}{|x-y|^{\mu}}~\mathrm{d}x\mathrm{d}y -2 \int_{\mb R^n \setminus B_{\delta}}\int_{B_{\delta}}\frac{|U_\epsilon(x)|^{2^*_{\mu}}|U_\epsilon(y)|^{2^*_{\mu}}}{|x-y|^{\mu}}~\mathrm{d}x\mathrm{d}y\\
& \quad \quad- \int_{\mb R^n \setminus B_{\delta}}\int_{\mb R^n \setminus B_{\delta}}\frac{|U_\epsilon(x)|^{2^*_{\mu}}|U_\epsilon(y)|^{2^*_{\mu}}}{|x-y|^{\mu}}~\mathrm{d}x\mathrm{d}y.
\end{split}
\end{equation}
We estimate the integrals in right hand side of \eqref{har-esti3} separately. Firstly to estimate the first integral, by Lemma \ref{S-H-minimizers} we get that $\{U_\e\}$ forms minimizers of $S_s^H$. Therefore using \eqref{eu-lag} we get
\begin{equation}\label{har-esti4}
{\int_{\mb R^n}\int_{\mb R^n}\frac{|U_\epsilon(x)|^{2^*_{\mu}}|U_\epsilon(y)|^{2^*_{\mu}}}{|x-y|^{\mu}}~\mathrm{d}x\mathrm{d}y
= \left(\frac{\|U_\epsilon\|^2}{S_s^H}\right)^{{2^*_\mu}}= \left(\frac{S_s^{n/2s}}{S_s^H}\right)^{{2^*_\mu}}= C(n,\mu)^{n/2s}(S_s^H)^{\frac{2n-\mu}{2s}}}
\end{equation}
Secondly, consider
\begin{equation*}\label{har-esti5}
\begin{split}
&\int_{\mb R^n \setminus B_{\delta}}\int_{B_{\delta}}\frac{|U_\epsilon(x)|^{2^*_{\mu}}|U_\epsilon(y)|^{2^*_{\mu}}}{|x-y|^{\mu}}~\mathrm{d}x\mathrm{d}y\\
&\leq C_{2,s}\int_{\mb R^n \setminus B_{\delta}}\int_{B_{\delta}} \frac{\epsilon^{\mu-2n}}{|x-y|^{\mu}\left( 1+ |\frac{x}{\epsilon}|^2 \right)^{\frac{2n-\mu}{2}}\left( 1+ |\frac{y}{\epsilon}|^2 \right)^{\frac{2n-\mu}{2}}}~\mathrm{d}x\mathrm{d}y\\
& = \epsilon^{2n-\mu}C_{2,s}\int_{\mb R^n \setminus B_{\delta}}\int_{B_{\delta}} \frac{1}{|x-y|^{\mu}\left( \epsilon^2+ |x|^2 \right)^{\frac{2n-\mu}{2}}\left( \epsilon^2+ |y|^2 \right)^{\frac{2n-\mu}{2}}}~\mathrm{d}x\mathrm{d}y.
\end{split}
\end{equation*}
where $C_{2,s}$ is an appropriate positive constant.{ Let $D:=  B_{\delta} \times (\mb R^n\setminus B_\delta )$  then
\begin{align*}
&\epsilon^{2n-\mu}C_{2,s}\int_{\mb R^n \setminus B_{\delta}}\int_{B_{\delta}} \frac{1}{|x-y|^{\mu}\left( \epsilon^2+ |x|^2 \right)^{\frac{2n-\mu}{2}}\left( \epsilon^2+ |y|^2 \right)^{\frac{2n-\mu}{2}}}~\mathrm{d}x\mathrm{d}y\\
&= \epsilon^{2n-\mu}C_{2,s}\left(\int_{D \cap \{|x-y|\leq 1\}}+ \int_{D \cap \{|x-y|>1\}} \right) \frac{1}{|x-y|^{\mu}\left( \epsilon^2+ |x|^2 \right)^{\frac{2n-\mu}{2}}\left( \epsilon^2+ |y|^2 \right)^{\frac{2n-\mu}{2}}}~\mathrm{d}x\mathrm{d}y.
\end{align*}
Consider
\begin{align*}
& \epsilon^{2n-\mu}C_{2,s}\int_{D \cap \{|x-y|>1\}}\frac{1}{|x-y|^{\mu}\left( \epsilon^2+ |x|^2 \right)^{\frac{2n-\mu}{2}}\left( \epsilon^2+ |y|^2 \right)^{\frac{2n-\mu}{2}}}~\mathrm{d}x\mathrm{d}y\\
& = \epsilon^{2n-\mu}C_{2,s}\int_{D \cap \{|x-y|>1\}}\frac{1}{\left( \epsilon^2+ |x|^2 \right)^{\frac{2n-\mu}{2}}\left( \epsilon^2+ |y|^2 \right)^{\frac{2n-\mu}{2}}}~\mathrm{d}x\mathrm{d}y\\
& \leq \epsilon^{2n-\mu}C_{2,s} \int_{B_\delta}\frac{\mathrm{d}x}{\left( \epsilon^2+ |x|^2 \right)^{\frac{2n-\mu}{2}}} \int_{\mb R^n\setminus B_\delta}\frac{\mathrm{d}y}{\left( \epsilon^2+ |y|^2 \right)^{\frac{2n-\mu}{2}}}\\
& \leq \epsilon^{2n-\mu}C_{2,s} \int_0^{\delta/\e}\frac{\e^{\mu-n}t^{n-1}\mathrm{d}t}{(1+t^2)^{2n-\mu}}\int_{\mb R^n\setminus B_\delta}\frac{\mathrm{d}y}{\left( |y|^2 \right)^{\frac{2n-\mu}{2}}}= O(\e^n).
\end{align*}
Next we observe that the set $D\cap  \{|x-y|>1\} $ is bounded and if $x, y \in D\cap  \{|x-y|>1\}$ then there exist constants $\alpha, \beta >0$ such that $\alpha \leq |x|,|y|\leq \beta$. This implies that
\begin{align*}
& \epsilon^{2n-\mu}C_{2,s}\int_{D \cap \{|x-y|>1\}}\frac{1}{|x-y|^{\mu}\left( \epsilon^2+ |x|^2 \right)^{\frac{2n-\mu}{2}}\left( \epsilon^2+ |y|^2 \right)^{\frac{2n-\mu}{2}}}~\mathrm{d}x\mathrm{d}y\\
& \leq \epsilon^{2n-\mu}C_{2,s}\int_{D \cap \{|x-y|>1\}}\frac{1}{|x-y|^{\mu}\left( |x|^2 \right)^{\frac{2n-\mu}{2}}\left( |y|^2 \right)^{\frac{2n-\mu}{2}}}~\mathrm{d}x\mathrm{d}y\\
& \leq O( \epsilon^{2n-\mu})\int_{D \cap \{|x-y|>1\}}\frac{1}{|x-y|^{\mu}}~\mathrm{d}x\mathrm{d}y= O(\e^{2n-\mu})
\end{align*}
since $\mu \in (0,n)$. Therefore
\begin{equation}\label{har-esti5}
\int_{\mb R^n \setminus B_{\delta}}\int_{B_{\delta}}\frac{|U_\epsilon(x)|^{2^*_{\mu}}|U_\epsilon(y)|^{2^*_{\mu}}}{|x-y|^{\mu}}~\mathrm{d}x\mathrm{d}y \leq O(\e^{2n-\mu}).
\end{equation}}

Lastly, in a similar manner we have
\begin{equation}\label{har-esti6}
\begin{split}
&\int_{\mb R^n \setminus B_{\delta}}\int_{\mb R^n \setminus B_{\delta}}\frac{|U_\epsilon(x)|^{2^*_{\mu}}|U_\epsilon(y)|^{2^*_{\mu}}}{|x-y|^{\mu}}~\mathrm{d}x\mathrm{d}y\\
&\leq C_{2,s}^\prime\int_{\mb R^n \setminus B_{\delta}}\int_{\mb R^n\setminus B_{\delta}} \frac{\epsilon^{\mu-2n}}{|x-y|^{\mu}\left( 1+ |\frac{x}{\epsilon}|^2 \right)^{\frac{2n-\mu}{2}}\left( 1+ |\frac{y}{\epsilon}|^2 \right)^{\frac{2n-\mu}{2}}}~\mathrm{d}x\mathrm{d}y\\
& = \epsilon^{2n-\mu}C_{2,s}^\prime\int_{\mb R^n \setminus B_{\delta}}\int_{\mb R^n \setminus B_{\delta}} \frac{1}{|x-y|^{\mu}\left( \epsilon^2+ |x|^2 \right)^{\frac{2n-\mu}{2}}\left( \epsilon^2+ |y|^2 \right)^{\frac{2n-\mu}{2}}}~\mathrm{d}x\mathrm{d}y\\
& \leq \epsilon^{2n-\mu}C_{2,s}^\prime \int_{\mb R^n \setminus B_{\delta}} \int_{ \mb R^n \setminus B_{\delta}} \frac{1}{|x-y|^\mu|x|^{2n-\mu}|y|^{2n-\mu}}\mathrm{d}x\mathrm{d}y = O(\epsilon^{2n-\mu}).
\end{split}
\end{equation}
where $ C_{2,s}^\prime$ is an appropriate positive constant. Using the estimates \eqref{har-esti4}, \eqref{har-esti5} and \eqref{har-esti6} in \eqref{har-esti3}, we get
 \begin{equation*}\label{har-esti7}
\left( \int_{\Om}\int_{\Om}\frac{|u_\epsilon(x)|^{2^*_{\mu}}|u_\epsilon(y)|^{2^*_{\mu}}}{|x-y|^{\mu}}~\mathrm{d}x\mathrm{d}y \right)^{\frac{n-2s}{2n-\mu}} \geq \left( (C(n,\mu))^{\frac{n}{2s}}(S^H_s)^{\frac{2n-\mu}{2s}}- O(\epsilon^{{2n-\mu} })\right)^{\frac{n-2s}{2n-\mu}}.
 \end{equation*}
 This completes the proof.{\hfill{\QED}}
\end{proof}

\section{Analysis of fibering maps}
In this section we study the fibering maps and establish some preliminaries for the Nehari manifold. It is easy to see that the energy functional $I_{\la,\delta}$ is not bounded below on the whole domain $Y$, so we minimize $I_{\la,\delta}$ over proper subsets of the Nehari manifold.
%
We define the set
\[\mc N_{\la,\delta}:= \{(u,v)\in Y\setminus \{0\}:\; (I_{\la,\delta}^\prime(u,v),(u,v))=0\}\]
and find that the functional $I_{\la,\delta}$ is bounded below on $\mc N_{\la,\delta}$. Therefore we state the following Lemma without giving the proof.
\begin{Lemma}\label{func-bdd-below}
$I_{\la,\delta}$ is coercive and bounded below on $\mc N_{\la,\delta}$ for any $\la,\delta>0$.
\end{Lemma}
{\begin{proof}
Let $\la, \delta >0$ and $(u,v) \in \mc N_{\la,\delta}$. Then it holds that
\begin{equation*}
\begin{split}
I_{\la,\delta}(u,v) & = {C_s^n}\left(\frac12 - \frac{1}{22^*_\mu} \right) \|(u,v)\|^2 - \left( \frac{1}{q} - \frac{1}{22^*_\mu}\right) \q\\
& \geq {C_s^n}\left(\frac12 - \frac{1}{22^*_\mu} \right) \|(u,v)\|^2 - \left( \frac{1}{q} - \frac{1}{22^*_\mu}\right)|\Om|^{\frac{2^*_s-q}{2^*_s}} (\la^{\frac{2}{2-q}} + \delta^{\frac{2}{2-q}}) S_s^{-\frac{q}{2}}\|(u,v)\|^q
\end{split}
\end{equation*}
and this yields the assertion because $1<q<2$. \hfill{\QED}
\end{proof}}

\noi From the definition of $\mc N_{\la,\delta}$, it is obvious that $(u,v)\in \mc N_{\la,\delta}$ if and only if $(u,v)\neq (0,0)$ and
\[{C_s^n}\|(u,v)\|^2 = \int_{\Om}(\la|u|^q+\delta|v|^q)\mathrm{d}x+ 2 B(u,v).\]
Let us define the fibering map $\varphi_{u,v}:\mb R^+ \to \mb R$ as
\[\varphi_{u,v}(t)= I_{\la,\delta}(tu,tv)= \frac{t^2{C_s^n}}{2}\|(u,v)\|^2-\frac{t^q}{q}\q- \frac{t^{22^*_\mu}}{2^*_\mu}B(u,v).\]
This gives another characterization of $\mc N_{\la,\delta}$ as follows
\[\mc N_{\la,\delta}=\{(tu,tv)\in Y\setminus\{(0,0)\}:\; \varphi_{u,v}^\prime (t)=0\}
\]
because $\varphi_{u,v}^\prime(t)= (I_{\la,\delta}^\prime(tu,tv),(u,v))$. An easy computation yields
\begin{equation}\label{varphi'}
\varphi_{u,v}^\prime(t)= t{C_s^n}\|(u,v)\|^2-t^{q-1}\q-2t^{22^*_\mu-1}B(u,v)
\end{equation}
\begin{equation}\label{varphi''}
\text{and}\; \varphi_{u,v}^{\prime\prime}(t)= {C_s^n}\|(u,v)\|^2- (q-1)t^{q-2}\q-2(22^*_\mu-1)t^{22^*_\mu-2}B(u,v).
\end{equation}
If $(u,v)\in \mc N_{\la,\delta}$ then \eqref{varphi'} and \eqref{varphi''} gives
\begin{align*}\label{varphi''n}
\varphi^{\prime\prime}_{u,v}(1)&= (2-q){C_s^n}\|(u,v)\|^2+ 2(q-22^*_\mu)B(u,v)\\
&= (2-22^*_\mu){C_s^n}\|(u,v)\|^2+(22^*_\mu-q)\q.
\end{align*}
Naturally, our next step is to divide $\mc N_{\la,\delta}$ into three subsets corresponding to local minima, local maxima and point of inflexion of $\varphi_{u,v}$ namely
\begin{align*}
\mc N_{\la,\delta}^\pm := \{(u,v)\in \mc N_{\la,\delta}:\; \varphi_{u,v}^{\prime\prime}(1)\gtrless 0\}\;\; \text{and}\;\;\mc N_{\la,\delta}^0 := \{(u,v)\in \mc N_{\la.\delta}:\;  \varphi_{u,v}^{\prime\prime}(1)=0\}.
\end{align*}
Our next lemma says that the local minimizers of $I_{\la,\delta}$ on the Nehari manifold $\mc N_{\la,\delta}$ are actually its critical points. So it is enough to prove the existence of minimizers of $I_{\la,\delta}$ on $\mc N_{\la,\delta}$.

\begin{Lemma}\label{min-is-sol}
Let $(u_1,v_1)$ and $(u_2,v_2)$ are minimizers of $I_{\la,\delta}$ on $\mc N_{\la,\delta}^+$ and $\mc N_{\la,\delta}^-$ respectively. Then $(u_1,v_1)$ and $(u_2,v_2)$ are nontrivial weak solutions of $(P_{\la,\delta})$.
\end{Lemma}
\begin{proof}
Let $(u_1,v_1) \in \mc N_{\la,\delta}^+$ such that $I_{\la,\delta}(u_1,v_1)= \inf I_{\la,\delta}(\mc N_{\la,\delta}^+)$ and define $V:= \{(u,v)\in Y: \; (J_{\la,\delta}^\prime(u,v),(u,v))>0\}$ where $J_{\la,\delta}(u,v)= ( I_{\la,\delta}^\prime(u,v),(u,v)) $. So, $\mc N_{\la,\delta}^+ = \{(u,v)\in V:\; J_{\la,\delta}(u,v)=0\}$ because for each $(u,v)$ such that $J_{\la,\delta}(u,v)=0$, we have $(J_{\la,\delta}^\prime(u,v),(u,v))>0$ if and only if $\varphi_{u,v}^{\prime \prime}(1)>0$. Therefore there exists Lagrangian multiplier $\rho \in \mb R$ such that
\[I_{\la,\delta}^\prime(u_1,v_1)= \rho J_{\la,\delta}^\prime(u_1,v_1).\]
Since $(u_1,v_1) \in \mc N_{\la,\delta}^+$, $(I_{\la,\delta}^\prime(u_1,v_1),(u_1,v_1))=0$ and $(J_{\la,\delta}^\prime(u_1,v_1),(u_1,v_1))>0$. This implies $\rho=0$. Therefore, $(u_1,v_1)$ is a nontrivial weak solution of $(P_{\la,\delta})$. Similarly, we can prove that if $(u_2,v_2) \in \mc N_{\la,\delta}^-$ is such that $I_{\la,\delta}(u_2,v_2)= \inf I_{\la,\delta}(\mc N_{\la,\delta}^-)$ then $(u_2,v_2)$ is also a nontrivial weak solution of $(P_{\la,\delta})$.\hfill{\QED}
\end{proof}


\noi For fixed $(u,v)\in Y\setminus\{(0,0)\}$, we write $\varphi^\prime_{u,v}(t)= t^{22^*_\mu-1}(m_{u,v}(t)- 2 B(u,v))$ where
\[m_{u,v}(t):= t^{2-22^*_\mu}{C_s^n}\|(u,v)\|^2 -t^{q-22^*_\mu}\q. \]
Clearly, $\varphi_{u,v}^\prime(t)=0$ if and only if $m_{u,v}(t)=2B(u,v)$ if and only if $(tu,tv)\in \mc N_{\la,\delta}$. So in order to understand the fibering maps, we study the map $m_{u,v}$. Since $2<22^*_\mu$ and $1<q<2$, we get
\[\lim_{t\to 0^+}m_{u,v}(t)=-\infty \; \text{and}\; \lim_{t\to +\infty}m_{u,v}(t)= 0. \]
\textbf{Claim:} The map $m_{u,v}(t)$ has a unique critical point at
\[t_{\text{max}}(u,v):= \left( \displaystyle\frac{\displaystyle(22^*_\mu-q)\q}{(22^*_\mu-2){C_s^n}\|(u,v)\|^2} \right)^{\frac{1}{2-q}}. \]
This follows from
\[m_{u,v}^\prime(t)= (2-22^*_\mu)t^{1-22^*_\mu}{C_s^n}\|(u,v)\|^2- (q-22^*_\mu)t^{q-1-22^*_\mu}\q. \]
We can check that $t_{\text{max}}(u,v)$ solves the equation $m_{u,v}^\prime(t)=0$. Also we can verify that since $1<q<2$
\begin{align*}
m_{u,v}^{\prime\prime}(t_{\text{max}}(u,v)) = \frac{(q-2)(22^*_\mu-2)^{\frac{2+22^*_\mu-q}{2-q}}{({C_s^n}\|(u,v)\|^2)^{\frac{2+22^*_\mu-q}{2-q}}}}{(22^*_\mu-q)^{\frac{22^*_\mu}{2-q}}
\left(\q\right)^{\frac{22^*_\mu}{2-q}}}<0
\end{align*}
implying that $t_{\text{max}}(u,v)$ is the point of maximum for the map $m_{u,v}(t)$. The uniqueness of the critical point of $m_{u,v}$ at $t_{\text{max}}(u,v)$ guarantees that $m_{u,v}(t)$ is strictly increasing in $(0,t_{\text{max}}(u,v))$ and strictly decreasing in $(t_{\text{max}}(u,v),+\infty)$.
If ${(tu,tv)}\in \mc N_{\la,\delta}$ then
\[t^{22^*_\mu-1}m_{u,v}^\prime(t)= \varphi_{u,v}^{\prime \prime}(t)= t^{-2}{\varphi_{tu,tv}^{\prime \prime}}(1)\]
which implies that $(tu,tv)\in \mc N_{\la,\delta}^{+}$(or $\mc N_{\la,\delta}^{-}$) if and only if $m_{u,v}^\prime(t)> 0$(or $m_{u,v}^\prime(t)< 0$).

\begin{Lemma}\label{Theta-def-lem}
For every $(u,v)\in Y \setminus \{(0,0)\}$ and $\la,\delta $ satisfying $0<\la^{\frac{2}{2-q}}+ \delta^{\frac{2}{2-q}} < \Theta$,
where
\begin{equation}\label{Theta-def}
\Theta := \left[ \frac{2^{2^*_\mu-1}{(C_s^n)^{\frac{22^*_\mu-q}{2-q}}}}{C(n,\mu)}\left(\frac{2-q}{22^*_\mu-q}\right) \left( \frac{22^*_\mu-2}{22^*_\mu-q}\right)^{\frac{22^*_\mu-2}{2-q}} S_s^{\frac{q(2^*_\mu-1)}{2-q}+2^*_\mu}|\Om|^{-\frac{(2^*_s-q)(22^*_\mu-2)}{2^*_s(2-q)}}\right]^{\frac{1}{2^*_\mu-1}},
\end{equation}
there exists unique $t_1,t_2>0$ such that $t_1<t_{\text{max}}(u,v)<t_2$, $(t_1u,t_1v) \in \mc N_{\la,\delta}^+$ and $(t_2u,t_2v)\in \mc N_{\la,\delta}^-$. Moreover,
\[I_{\la,\delta}(t_1u,t_1v)=\inf_{t \in [0,t_{\text{max}(u,v)}]} I_{\la,\delta}(tu,tv)\; \text{and}\; I_{\la,\delta}(t_2u,t_2v)=\sup_{t\geq 0} I_{\la,\delta}(tu,tv).\]
\end{Lemma}
\begin{proof}
Let $(u,v)\in Y\setminus \{(0,0)\}$. Then we have already seen that
\begin{equation}\label{Theta-def1}
m_{u,v}(t)=2B(u,v)
\end{equation}
 if and only if $(tu,tv)\in \mc N_{\la,\delta}$. Since $B(u,v)>0$, we say that \eqref{Theta-def1} can never hold if we choose $\la$ and $\delta$ such that
 $2B(u,v)>m_{u,v}(t_{\text{max}}(u,v))$
 and vice-versa.
 In this case, $(u,v) \not\in \mc N_{\la,\delta}$ and hence not a weak solution to $(P_{\la,\delta})$. Using H\"{o}lder's inequality and the definition of $S_s$, we get
 \begin{equation}\label{Theta-def2}
 \begin{split}
\q \leq S_s^{-\frac{q}{2}}|\Om|^{\frac{2^*_s-q}{2^*_s}}(\la \|u\|^q+ \delta \|v\|^q)
  { \leq S_s^{-\frac{q}{2}}|\Om|^{\frac{2^*_s-q}{2^*_s}}\|(u,v)\|^q(\la^{\frac{2}{2-q}}+\delta^{\frac{2}{2-q}})^{\frac{2-q}{2}}}.
 \end{split}
 \end{equation}

 Also from the definition of $\tilde S^H_s$ and Lemma \ref{reltn}, we get
 \begin{equation}\label{Theta-def3}
2 B(u,v)\leq  2(\tilde S_s^H)^{-2^*_\mu}\|(u,v)\|^{22^*_\mu}= 2^{1-2^*_\mu}S_s^{-2^*_\mu}C(n,\mu)\|(u,v)\|^{22^*_\mu}.
 \end{equation}
 Using \eqref{Theta-def2} we can estimate $m_{u,v}(t_{\text{max}})$ as follows
 \begin{equation}\label{Theta-def4}
 \begin{split}
 m_{u,v}({t_{\text{max}}(u,v)}) &= \left[\left(\frac{22^*_\mu-2}{22^*_\mu-q}\right)^{\frac{22^*_\mu-2}{2-q}}- \left(\frac{22^*_\mu-2}{22^*_\mu-q}\right)^{\frac{22^*_\mu-q}{2-q}}\right]\frac{({C_s^n}\|(u,v)\|^2)^{\frac{22^*_\mu-q}{2-q}}}{\left(\q\right)^{\frac{22^*_\mu-2}{2-q}}}\\
 & = \left(\frac{22^*_\mu-2}{22^*_\mu-q}\right)^{\frac{22^*_\mu-2}{2-q}}\left(\frac{2-q}{22^*_\mu-q} \right)\frac{({C_s^n}\|(u,v)\|^2)^{\frac{22^*_\mu-q}{2-q}}}{\left(\q\right)^{\frac{22^*_\mu-2}{2-q}}}\\
 & \geq \left(\frac{22^*_\mu-2}{22^*_\mu-q}\right)^{\frac{22^*_\mu-2}{2-q}}\left(\frac{2-q}{22^*_\mu-q} \right)\frac{{(C_s^n)^{\frac{22^*_\mu-q}{2-q}}}\|(u,v)\|^{22^*_\mu}}{\left(\la^{\frac{2}{2-q}}+ \delta^{\frac{2}{2-q}} \right)^{2^*_\mu-1}{\left({S_s}\right)^{-\frac{q(2^*_\mu-1)}{2-q}}}|\Om|^{\frac{(2^*_s-q)(22^*_\mu-2)}{2^*_s(2-q)}}}.
 \end{split}
 \end{equation}
 Now if $\la$ and $\delta$ satisfies $0< \la^{\frac{2}{2-q}}+ \delta^{\frac{2}{2-q}}<\Theta $, where $\Theta$ is given in \eqref{Theta-def}, then
 \begin{equation}\label{Theta-def5}
2^{1-2^*_\mu}S_s^{-2^*_\mu}C(n,\mu) \leq \left(\frac{22^*_\mu-2}{22^*_\mu-q}\right)^{\frac{22^*_\mu-2}{2-q}}\left(\frac{2-q}{22^*_\mu-q} \right)\frac{(C_s^n)^{\frac{22^*_\mu-q}{2-q}}}{\left(\la^{\frac{2}{2-q}}+ \delta^{\frac{2}{2-q}} \right)^{2^*_\mu-1}S_s^{-\frac{q(2^*_\mu-1)}{2-q}}|\Om|^{\frac{(2^*_s-q)(22^*_\mu-2)}{2^*_s(2-q)}}}
 \end{equation}
 which along with \eqref{Theta-def4} implies that
 \begin{equation}\label{Theta-def6}
 \begin{split}
 0<2B(u,v)< 2^{1-2^*_\mu} S_s^{-2^*_\mu} C(n,\mu)\|(u,v)\|^{22^*_\mu} < m_{u,v}(t_{\text{max}}(u,v)).
 \end{split}
 \end{equation}
 Therefore there exist unique $t_1,t_2>0$ with $t_1<t_{\text{max}}(u,v) <t_2$ such that
 \[m_{u,v}(t_1)= m_{u,v}(t_2)= 2B(u,v)\]
 and $m_{u,v}^\prime(t_1)>0$ and $m_{u,v}^\prime(t_1)<0$. This implies $(t_1u,t_1v)\in \mc N_{\la,\delta}^+$ and $ (t_2u,t_2v) \in \mc N_{\la,\delta}^-$ and also $\varphi_{u,v}^{\prime\prime}(t_1)>0$ and $\varphi_{u,v}^{\prime\prime}(t_2)<0$. From the definition of $\varphi_{u,v}$, we get
 \[I_{\la,\delta}(t_2u,t_2v)\geq I_{\la,\delta}(tu,tv)\geq I_{\la,\delta}(t_1u,t_1v)\; \text{for each}\; t \in[t_1,t_2];\]
 \[I_{\la,\delta}(t_1u,t_1v) \leq I_{\la,\delta}(tu,tv)\; \text{for each}\; t\in[0,t_1].\]
 Thus
 \[I_{\la,\delta}(t_1u,t_1v)=\inf_{t \in [0,t_{\text{max}}(u,v)]} I_{\la,\delta}(tu,tv)\; \text{and}\; I_{\la,\delta}(t_2u,t_2v)=\sup_{t\geq 0} I_{\la,\delta}(tu,tv).\]
 holds true.\hfill{\QED}
\end{proof}

\noi  We end this section with the following important lemma.
\begin{Lemma}\label{N_0-empty}
If $0 < \la^{\frac{2}{2-q}} + \delta^{\frac{2}{2-q} }< \Theta$, where $\Theta$ is as in \eqref{Theta-def} then $\mc N_{\la,\delta}^0= \emptyset$.
\end{Lemma}
\begin{proof}
We prove this by contradiction, so let $(u,v) \in \mc N_{\la,\delta}^0$. By Lemma \ref{Theta-def-lem} we know that there exist $t_1,t_2>0$ such that $\varphi^{\prime}_{u,v}(t_1)= 0 = \varphi^{\prime}_{u,v}(t_2)$ and $\varphi^{\prime \prime}_{u,v}(t_1)>0>\varphi^{\prime \prime}_{u,v}(t_2)$. But $(u,v) \in \mc N_{\la,\delta}^0$ means that $\varphi^{\prime \prime}_{u,v}(1)=0=\varphi^{\prime}_{u,v}(1)$. This is possible when either $t_1=1$ or $t_2=1$. But this again implies that $\varphi^{\prime \prime}_{u,v}(1)>0$ or $\varphi^{\prime \prime}_{u,v}(1)<0$, a contradiction. \hfill{\QED}
\end{proof}

\section{Existence of minimizers on $\mc N_{\la,\delta}^+$ and $\mc N_{\la,\delta}^-$}
Lastly, in this section we present the proof of Theorem \ref{MT}. We divide this section into two subsections where we prove existence of first and second solutions respectively.
\begin{Lemma}\label{PS-seq-bdd}
Let $\{(u_k,v_k)\}\subset Y$ be a $(PS)_c$ sequence that is
\[I_{\la,\delta}(u_k,v_k)\to c\; \text{in}\; \mb R \; \text{and}\; I_{\la,\delta}^\prime(u_k,v_k) \to 0\; \text{in}\; Y^*\; \text{as}\; k \to \infty. \]
Then $\{u_k,v_k\}$ is bounded in $Y$.
\end{Lemma}
\begin{proof}
Let $\{(u_k,v_k)\}\subset Y$ be a $(PS)_c$ sequence for $I_{\la,\delta}$ such that
\[I_{\la,\delta}(u_k,v_k)\to c\; \text{in}\; \mb R \; \text{and}\; I_{\la,\delta}^\prime(u_k,v_k) \to 0\; \text{in}\; Y^*\; \text{as}\; k \to \infty. \]
This can be equivalently written as
\begin{equation}\label{PS-seq-bdd1}
\frac{{C_s^n}}{2}\|(u_k,v_k)\|^2-\frac{1}{q}\Q -\frac{1}{2^*_\mu}B(u_k,v_k)= c+o_k(1),
\end{equation}
\begin{equation}\label{PS-seq-bdd2}
{C_s^n}\|(u_k,v_k)\|^2-\Q -2B(u_k,v_k)=o_k(\|(u_k,v_k)\|)
\end{equation}
as $k \to \infty$. We show the boundedness of the sequence $\{(u_k,v_k)\}$ in $Y$ using the method of contradiction. So assume, on contrary, $\|(u_k,v_k)\|\to \infty$ as $k \to \infty$ and set
\[w_k:= \frac{u_k}{\|(u_k,v_k)\|}, \;\; z_k:= \frac{v_k}{\|(u_k,v_k)\|}.\]
Clearly, $\|(w_k,z_k)\|=1$, for all $k$ which implies that there exists a subsequence, still denoted by $\{(w_k,z_k)\}$, such that $(w_k,z_k)\rightharpoonup (w,z)$ weakly in $Y$ as $k\to \infty$, for some $(w,z)\in Y$. By fractional Sobolev embedding results,
 we get
\begin{equation}\label{PS-seq-bdd3}
\int_{\Om}(\la|w_k|^q+ \delta|z_k|^q)\mathrm{d}x \to \int_{\Om}(\la|w|^q+ \delta|z|^q)\mathrm{d}x \; \text{as}\; k \to \infty.
\end{equation}
Putting $u_k= w_k\|(u_k,v_k)\|$ and $v_k=z_k\|(u_k,v_k)\|$ in \eqref{PS-seq-bdd1} and \eqref{PS-seq-bdd2} and solving we get
\[\frac{{C_s^n}}{2} \|(w_k,z_k)\|^2- \frac{\|(u_k,z_k)\|^{{q-2}}}{q}\int_{\Om}(\la|w_k|^q+ \delta|z_k|^q)\mathrm{d}x- \frac{1}{2^*_\mu}\|(u_k,v_k)\|^{22^*_\mu-2}B(w_k,z_k)=o_k(1), \]
\[{C_s^n}\|(w_k,z_k)\|^2-\|(u_k,v_k)\|^{q-2}\int_{\Om}(\la|w_k|^q+ \delta|z_k|^q)\mathrm{d}x -2\|(u_k,v_k)\|^{22^*_\mu-2}B(w_k,z_k)=o_k(1). \]
From above these two equations and \eqref{PS-seq-bdd3}, we get
\begin{align*}
{C_s^n}\|(w_k,z_k)\|^2 &= \frac{(22^*_{\mu}-q)}{q(2^*_\mu-1)} \|(u_k,v_k)\|^{q-2} \int_{\Om}(\la|w_k|^q+ \delta|z_k|^q)\mathrm{d}x + o_k(1)\\
& = \frac{(22^*_{\mu}-q)}{q(2^*_\mu-1)} \|(u_k,v_k)\|^{q-2} \int_{\Om}(\la|w|^q+ \delta|z|^q)\mathrm{d}x + o_k(1).
\end{align*}
Since $1<q<2$ and $\|(u_k,v_k)\|\to \infty$ we get $\|(w_k,z_k)\|^2 \to 0$ as $k\to \infty$ which contradicts $\|(w_k,z_k)\|=1$ for all $k$. This completes the proof. \hfill{\QED}
\end{proof}

\begin{Lemma}\label{unif-lwr-bd}
If  $\{(u_k,v_k)\}$ is a $(PS)_c$ sequence for $I_{\la,\delta}$ with $(u_k,v_k) \rightharpoonup (u,v)$ weakly in $Y$ as $k \to \infty$, then $I_{\la,\delta}^\prime(u,v)=0$. Moreover there exists a positive constant $D_0$ depending on $\mu,q,s,n,S_s$ and $\Om$ such that
\begin{equation}\label{unif-lwr-bd1}
I_{\la,\delta}(u,v) \geq -D_0(\la^{\frac{2}{2-q}}+ \delta^{\frac{2}{2-q}}),
\end{equation}
where
\[D_0:= {\frac{(2-q)(22^*_\mu-q)}{42^*_\mu q}\left[\left( \frac{22^*_\mu {C_s^n}S_s(n-\mu+2s)}{(2n-\mu)(22^*_\mu-q)}\right)^{-\frac{q}{2}}|\Om|^{\frac{2^*_s-q}{2^*_s}}\right]^{\frac{2}{2-q}}}. \]
\end{Lemma}
\begin{proof}
Let $\{(u_k.v_k)\} \subset Y$ be a $(PS)_c$ sequence for $I_{\la,\delta}$ such that $(u_k,v_k) \rightharpoonup (u,v)$ weakly in $Y$ as $k \to \infty$. This implies $I_{\la,\delta }^\prime (u_k,v_k) = o_k(1)\; \text{in} \; Y^*\; \text{as}\; k \to \infty$.
Let $(\phi,\psi) \in Y$. From weak convergence it follows that
\begin{equation}\label{unif-lwr-bd5}
\lim_{k \to \infty} \langle u_k,\phi\rangle =\langle u,\phi \rangle \; \text{and}\;  \lim_{k \to \infty} \langle v_k,\psi\rangle = \langle v,\psi \rangle.
\end{equation}

\noi For $q^\prime = \displaystyle\frac{q}{q-1}$  we also have
\begin{equation}\label{unif-lwr-bd3}
\begin{split}
|u_k|^{q-2}u_k \rightharpoonup |u|^{q-2}u, \; |v_k|^{q-2}v_k \rightharpoonup |v|^{q-2}v\; \text{in}\; L^{q^\prime}(\Om) \; \text{and}\;u_k \rightharpoonup u,\;  v_k \rightharpoonup v\; \text{in}\; L^{2^*_s}(\Om).
\end{split}
\end{equation}
as $k \to \infty$, thanks to the embedding of $X_0$ into $L^m(\Om)$ for all $1 \leq  m \leq 2^*_s$.
Since we assumed $\phi,\psi \in X_0$ which is contained in $ L^{q}(\Om)\cap L^{2^*_s}(\Om)$, so from \eqref{unif-lwr-bd3} it follows that as $k \to \infty$
\begin{equation}\label{unif-lwr-bd6}
\int_{\Om}|u_k|^{q-2}u_k\phi\mathrm{d}x \to \int_{\Om}|u|^{q-2}u\phi\mathrm{d}x.
\end{equation}
Also since $2^*_\mu-1 = \displaystyle\frac{n-\mu+2s}{n-2s}$  and $|u_k|^{2^*_\mu} \rightharpoonup  |u|^{2^*_\mu}$, $ |v_k|^{2^*_\mu}\rightharpoonup  |v|^{2^*_\mu} \; \text{in}\; L^{\frac{2n}{2n-\mu}}(\Om)$, we get
\[|u_k|^{2^*_\mu-2}u_k \rightharpoonup  |u|^{2^*_\mu-2}u\; \text{and} |v_k|^{2^*_\mu-2}v_k \rightharpoonup  |v|^{2^*_\mu-2}v \; \text{in}\; L^{\frac{2n}{n-\mu+2s}}(\Om).\]
By Hardy-Littlewood-Sobolev inequality, {the Riesz potential defines a linear and continuous map from $L^{\frac{2n}{2n-\mu}}(\Om)$ to $L^{\frac{2n}{\mu}}(\Om)$ which gives}
\begin{equation}\label{unif-lwr-bd9}
|x|^{-\mu}\ast|u_k|^{2^*_\mu} \rightharpoonup |x|^{-\mu}\ast|u|^{2^*_\mu} \; \text{and}\; |x|^{-\mu}\ast|v_k|^{2^*_\mu} \rightharpoonup |x|^{-\mu}\ast|v|^{2^*_\mu} \; \text{in}\; L^{\frac{2n}{\mu}}(\Om).
\end{equation}
{This implies that the sequences $(|x|^{-\mu}\ast |u_k|^{2^*_\mu})|v_k|^{2^*_\mu-2}v_k$ and $(|x|^{-\mu}\ast |v_k|^{2^*_\mu})|u_k|^{2^*_\mu-2}u_k$ converges weakly in $L^{\frac{2n}{n+2s}}(\Om)$. Through Sobolev embedding we know that
\begin{equation}\label{unif-lwr-bd10}
|u_k|^{2^*_\mu-2}u_k \to  |u|^{2^*_\mu-2}u\; \text{and} |v_k|^{2^*_\mu-2}v_k \to |v|^{2^*_\mu-2}v \; \text{in}\; L^{\frac{2n}{2n-\mu}}(\Om)
\end{equation}
Taking into account \eqref{unif-lwr-bd9} and \eqref{unif-lwr-bd10}, for any $\tilde\psi \in L^\infty(\Om)$ we obtain
\begin{align*}
\int_{\Om}(|x|^{-\mu}\ast |u_k|^{2^*_\mu})|v_k|^{2^*_\mu-2}v_k \tilde\psi~\mathrm{d}x &\to \int_{\Om}(|x|^{-\mu}\ast |u|^{2^*_\mu})|v|^{2^*_\mu-2}v\tilde\psi~\mathrm{d}x \\
\; \text{and}\; \int_{\Om} (|x|^{-\mu}\ast |v_k|^{2^*_\mu})|u_k|^{2^*_\mu-2}u_k\tilde\psi~\mathrm{d}x & \to \int_{\Om}(|x|^{-\mu}\ast |v|^{2^*_\mu})|u|^{2^*_\mu-2}u\tilde\psi~\mathrm{d}x .
\end{align*}
Therefore the sequences $(|x|^{-\mu}\ast |u_k|^{2^*_\mu})|v_k|^{2^*_\mu-2}v_k$ and $(|x|^{-\mu}\ast |v_k|^{2^*_\mu})|u_k|^{2^*_\mu-2}u_k$ converges in the distributional sense. Since the  weak limit and the distributional limit coincides, for $\phi,\psi \in X_0(\Om)\subset L^{2^*_s}(\Om)$, we get that as $k \to \infty$
\begin{equation}\label{unif-lwr-bd4}
\begin{split}
\int_\Om (|x|^{-\mu}\ast |u_k|^{2^*_\mu})|v_k|^{2^*_\mu-2}v_k\psi~\mathrm{d}x &\to \int_\Om (|x|^{-\mu}\ast |u|^{2^*_\mu})|v|^{2^*_\mu-2}v\psi~\mathrm{d}x,\\
\int_\Om (|x|^{-\mu}\ast |v_k|^{2^*_\mu})|u_k|^{2^*_\mu-2}u_k\phi~\mathrm{d}x & \to \int_\Om(|x|^{-\mu}\ast |v|^{2^*_\mu})|u|^{2^*_\mu-2}u\phi~\mathrm{d}x.
\end{split}
\end{equation}}
So using \eqref{Iprime}, \eqref{unif-lwr-bd5}, \eqref{unif-lwr-bd6} and \eqref{unif-lwr-bd4}  we get $(I_{\la,\delta}^\prime(u_k,v_k)- I_{\la,\delta}^\prime(u,v), (\phi,\psi)) \to 0$ as $k \to \infty$, for all $(\phi,\psi) \in Y $ which implies that $I_{\la,\delta}^\prime(u,v) =0$. Therefore $(u,v)$ is a weak solution of $(P_{\la,\delta})$ and $(u,v) \in \mc N_{\la,\delta}$. That is
\[{C_s^n}\|(u,v)\|^2= \q +2 B(u,v) \]
which gives
\begin{equation}\label{unif-lwr-bd7}
I_{\la,\delta}(u,v) = \frac{(2^*_\mu-1){C_s^n}}{22^*_\mu}\|(u,v)\|^2 - \frac{22^*_\mu -q}{22^*_\mu q} \q.
\end{equation}
Let $D = \left[\displaystyle \frac{2}{q}\cdot\frac{n-\mu+2s}{2(2n-\mu)} \left( \frac{1}{q}- \frac{1}{22^*_\mu}\right)^{-1}\right]$
Using H\"{o}lder's inequality, fractional Sobolev inequality, definition of $S_s$ and Young's inequality we get the following estimate
\begin{equation}\label{unif-lwr-bd8}
\begin{split}
&\q \leq |\Om|^{\frac{2^*_s-q}{2^*_s}}S_s^{-\frac{q}{2}} (\la \|u\|^q + \delta \|v\|^q) \\
& = \left( D^{\frac{q}{2}}{(C_s^n)^{\frac{q}{2}}} \|u\|^q \right) \left( D^{-\frac{q}{2}} \la|\Om|^{\frac{2^*_s-q}{2^*_s}}({C_s^n}S_s)^{-\frac{q}{2}} \right)+ \left( D^{\frac{q}{2}} {(C_s^n)^{\frac{q}{2}}} \|v\|^q \right) \left( D^{-\frac{q}{2}} \delta |\Om|^{\frac{2^*_s-q}{2^*_s}} ({C_s^n}S_s)^{-\frac{q}{2}}\right) \\
& \leq \frac{n-\mu+2s}{2(2n-\mu)}  \left( \frac{1}{q}- \frac{1}{22^*_\mu}\right)^{-1} {C_s^n}(\|u\|^2 + \|v\|^2)+ \tilde D \left(\la^{\frac{2}{2-q}}+ \delta^{\frac{2}{2-q} }\right)\\
&=  \frac{n-\mu+2s}{2(2n-\mu)}  \left( \frac{1}{q}- \frac{1}{22^*_\mu}\right)^{-1}{C_s^n} \|(u,v)\|^2 + \tilde D \left(\la^{\frac{2}{2-q}}+ \delta^{\frac{2}{2-q} }\right),
\end{split}
\end{equation}
where $\tilde D := \displaystyle\frac{2-q}{2}\left(D^{-\frac{q}{2}}|\Om|^{\frac{2^*_s-q}{2^*_s}} ({C_s^n}S_s)^{-\frac{q}{2}}\right)^{\frac{2}{2-q}}$.
Using \eqref{unif-lwr-bd8} in \eqref{unif-lwr-bd7}, we finally obtain \eqref{unif-lwr-bd1} with $D_0 = \left(\displaystyle\frac{22^*_\mu-q}{22^*_\mu q}\right) \tilde D$. This completes the proof.\hfill{\QED}
\end{proof}

\noi As a consequence of Lemma \ref{N_0-empty} we infer that for any $\la, \delta$ satisfying $0 < \la^{\frac{2}{2-q}} + \delta^{\frac{2}{2-q} }< \Theta$,
\[\mc N_{\la,\delta}= \mc N_{\la,\delta}^+ \cup \mc N_{\la,\delta}^-. \]
In spirit of Lemma \ref{func-bdd-below}, we define the following
\[l_{\la,\delta}= \inf_{\mc N_{\la,\delta}} I_{\la,\delta} \; \text{and} \;l_{\la,\delta}^\pm = \inf_{\mc N_{\la,\delta}^\pm} I_{\la,\delta}. \]
Then we have the following result.

\begin{Lemma}\label{inf-pos-neg}
The following holds true:
\begin{enumerate}
 \item[(i)]If $0 < \la^{\frac{2}{2-q}} + \delta^{\frac{2}{2-q}}< \Theta$, then {$ l_{\la,\delta} \leq l_{\la,\delta}^+ < 0$},
 \item[(ii)] $\inf \{\|(u,v)\|:\; (u,v)\in \mc N_{\la,\delta}^- \}>0$ and $\sup\{\|(u,v)\|:\; (u,v)\in \mc N_{\la,\delta}^-, \; I_{\la,\delta}(u,v)\leq M \}< +\infty$ for each $M>0$.
     \end{enumerate}
\end{Lemma}
\begin{proof}
\begin{enumerate}
\item[(i)] {Let $(u,v)\in \mc N_{\la,\delta}^+$ implying that $\varphi_{u,v}^\prime(1)=0$ and $\varphi_{u,v}^{\prime \prime}(1)>0$. Therefore
\[\frac{(2-q){C_s^n}}{2(22^*_\mu-q)}\|(u,v)\|^2 > B(u,v). \]
Using this we deduce that
\begin{align*}
I_{\la,\delta}(u,v) &= \left(\frac12 - \frac{1}{q} \right){C_s^n}\|(u,v)\|^2 + \left( \frac{2}{q}- \frac{1}{2^*_\mu}\right)B(u,v)\\
& < \left( \frac{q-2}{2q} + \frac{2-q}{22^*_\mu q}\right){C_s^n}\|(u,v)\|^2 = \frac{2-q}{2q} \left(\frac{1}{2^*_\mu}-1\right)\|(u,v)\|^2 <0.
\end{align*}
This alongwith the definition of $l_{\la,\delta}$ and $l_{\la,\delta}^+$ implies that $l_{\la,\delta} \leq l_{\la,\delta}^+ <0$.}
\item[(ii)] Let $(u,v)\in \mc N_{\la,\delta}^-$ then using Lemma \ref{ineq} and \eqref{ineq1} we get
\begin{align*}
0> \varphi_{u,v}^{\prime \prime}(1) \geq {(2-q){C_s^n}\|(u,v)\|^2 - 2 (22^*_\mu- q)  ({\tilde S_s^H})^{-2^*_\mu}\|(u,v)\|^{22^*_\mu}}.
\end{align*}
This gives
\[\|(u,v)\|\geq {\left( \frac{(2-q){C_s^n}}{2(22^*_\mu-q)(S_s^H)^{-2^*_\mu}}\right)^{\frac{1}{22^*_\mu-2}}}>0 \]
{which implies that $\inf \{\|(u,v)\|:\; (u,v)\in \mc N_{\la,\delta}^- \}>0$}.
 Therefore $\inf \{\|(u,v)\|:\; (u,v)\in \mc N_{\la,\delta}^- \}>0$.
Now let $I_{\la,\delta}(u,v)\leq M$ for some $M>0$ then an easy computation yields
\[\left(\frac{1}{2}- \frac{1}{22^*_\mu}\right){C_s^n} \|(u,v)\|^2 - K_{\la,\delta}\left(\frac{1}{q}- \frac{1}{22^*_\mu}\right)\|(u,v)\|^q\leq M  \]
where $K_{\la,\delta}= S_s^{-\frac{q}{2}} |\Om|^{\frac{2^*_s-q}{2^*_s} }(\la+\delta)$ which completes the proof. \hfill{\QED}
\end{enumerate}
\end{proof}

\noi Our next result is established by using the implicit function theorem and it plays a crucial role in proving Theorem \ref{PS-seq}.

\begin{Proposition}\label{IFT-N}
Assume $0 < \la^{\frac{2}{2-q}} + \delta^{\frac{2}{2-q}}< \Theta$ and $w \;= \;(u,v) \in \mc N_{\la,\delta}$. Then there exist $\e>0$ and a differentiable function $\zeta : B_\e(0) \subset Y \to \mb R^+$ ($B_\e(0)$ denotes ball of radius $\e$ with center origin) such that $\zeta(0)=1$, $\zeta(z)(w-z) \in \mc N_{\la,\delta}$ and
\begin{equation}\label{IFT-N1}
(\zeta^\prime(0),z) = - \frac{2(\langle u,z_1\rangle + \langle v, z_2\rangle) - T_{\la,\delta}(w,z) - 2 M(z) }{(2-q){C_s^n} \|(u,v)\|^2- 2(22^*_\mu)B(u,v)}
\end{equation}
for all $z=(z_1,z_2) \in B_\e(0)$, where
\begin{equation*}\label{IFT-N2}
\begin{split}
T_{\la,\delta}(w,z)&= q\int_\Om (\la |u|^{q-2}uz_1 + \delta |v|^{q-2}v z_2)\mathrm{d}x,\\
M(z) & = \int_{\Om} ((|x|^{-\mu}\ast |v|^{2^*_\mu})|u|^{2^*_\mu-2}u z_1 + (|x|^{-\mu}\ast |u|^{2^*_\mu})|v|^{2^*_\mu-2}v z_2 )\mathrm{d}x.
\end{split}
\end{equation*}
\end{Proposition}
\begin{proof}
For $w = (u,v) \in \mc N_{\la,\delta}$, let us define $\mathfrak{F}_w: \mb R^+ \times Y \to \mb R^n$ by
\begin{align*}
\mathfrak{F}_w(\rho,z) &:= (I_{\la,\delta}^{\prime}(\rho(w-z)),(\rho(w-z)))\\
& = \rho^2 {C_s^n} \|(u-z_1, v-z_2)\|^2 - \rho^q \int_{\Om} (\la |u-z_1|^q+ \delta |v-z_2|^q)\mathrm{d}x - 2\rho^{22^*_\mu}B(u-z_1,v-z_2)
\end{align*}
where $\rho \in \mb R^+$ and $z=(z_1,z_2) \in Y$. Then clearly $\mathfrak{F}_w(1,(0,0)) = (I_{\la,\delta}^\prime(w),w)=0$ since $w \in \mc N_{\la,\delta}$. Also \begin{align*}
\frac{d}{d\rho}\mathfrak{F}_w(1,(0,0)) &= 2{C_s^n} \|(u,v)\|^2 -q \q - 2(22^*_\mu) B(u,v)\\
&= (2-q){C_s^n} \|(u,v)\|^2 - 2(22^*_\mu-q)\q  = \varphi_{u,v}^{\prime \prime}(1) \neq 0
\end{align*}
because of Lemma \ref{N_0-empty}. Therefore we can apply the implicit function theorem to obtain a $\e>0$ and a differentiable map $ \zeta: B_\e(0) \subset Y \to \mb R^+$ with $\zeta(0)=1$ and satisfies \eqref{IFT-N1}. Also $\mathfrak{F}_w(\zeta)=0$ for all $z \in B_\e(0)$ which is equivalent to
\[( I_{\la,\delta}^\prime(\zeta(z)(w-z)),\zeta(z)(w-z)) =0, \; \text{for all}\; z \in B_\e(0),\]
that is $\zeta(z)(w-z) \in \mc N_{\la,\delta}$. \hfill{\QED}
\end{proof}


\begin{Theorem}\label{PS-seq}
 If $0<  \la^{\frac{2}{2-q}} + \delta^{\frac{2}{2-q}}< \Theta$ then there exists a $(PS)_{l_{\la,\delta}}$ sequence $\{(u_k,v_k)\} \subset \mc N_{\la,\delta}$ for $I_{\la,\delta}$.
\end{Theorem}
\begin{proof}
 We use the Ekeland Variational principle to say that there exists a minimizing sequence $\{(u_k,v_k)\} \subset \mc N_{\la,\delta}$ such that \begin{equation}\label{PS-seq1}
    I_{\la,\delta}(u_k,v_k) < l_{\la,\delta} + \frac{1}{k} \quad \text{ and } \quad\  I_{\la,\delta}(u_k,v_k) < I_{\la,\delta}(w_1,w_2) + \frac{1}{k} \|(w_1,w_2)- (u_k,v_k)\|,
    \end{equation}
    for each $(w_1,w_2) \in \mc N_{\la,\delta}$. From Lemma \ref{inf-pos-neg}(i) we know that $l_{\la,\delta}<0$, therefore we can find $k$ sufficiently large such that
    \begin{equation}\label{PS-seq2}
    I_{\la,\delta}(u_k,v_k) = \left(\frac12- \frac{1}{22^*_\mu} \right){C_s^n}\|(u,v)\|^2 - \left(\frac{1}{q}- \frac{1}{22^*_\mu} \right)\q < \frac{l_{\la,\delta}}{2}.
    \end{equation}
    This gives us
    \begin{equation}\label{PS-seq3}
    - \frac{2^*_\mu q}{(22^*_\mu -q)}l_{\la,\delta} < \q < S_s^{-\frac{q}{2}} |\Om|^{\frac{2^*_s - q}{2^*_s}} (\la^{\frac{2}{2-q}} + \delta^{\frac{2}{2-q}})^{\frac{2-q}{2}} \|(u_k,v_k)\|^q.
    \end{equation}
    Consequently $(u_k,v_k) \neq 0$. From \eqref{PS-seq3} we get
    \begin{equation}\label{PS-seq4}
    \|(u_k,v_k)\| > \left( - \frac{2^*_\mu q l_{\la,\delta}}{22^*_\mu-q} S_s^{\frac{q}{2}} |\Om|^{-\frac{2^*_s-q}{2^*_s}} \left(\la^{\frac{2}{2-q}} + \delta^{\frac{2}{2-q}}\right)^{\frac{q-2}{2}} \right)^{\frac{1}{q}}
    \end{equation}
    and from \eqref{PS-seq2} we get
     \begin{equation}\label{PS-seq5}
    \|(u_k,v_k)\| < \left(  \frac{(22^*_\mu-q)}{q(2^*_\mu-1)} {S_s^{-\frac{q}{2}}} |\Om|^{\frac{2^*_s-q}{2^*_s}} \left(\la^{\frac{2}{2-q}} + \delta^{\frac{2}{2-q}}\right)^{\frac{2-q}{2}} \right)^{\frac{1}{2-q}}.
    \end{equation}
    \textbf{Claim:} $I_{\la,\delta}^\prime(u_k,v_k) \to 0$ in $Y^*$ as $k \to \infty$.\\
    Let us fix $k \in \mb N$ then by applying {Proposition} \ref{IFT-N} to $w_k = (u_k,v_k)$, we get that there exists a function $\zeta_k : B_{\e_k}(0) \to \mb R^+$ for some $\e_k >0$ such that $\zeta_k(h)(w_k-h) \in \mc N_{\la,\delta}$ for $h=(h_1,h_2) \in {B_{\e_k}}(0)$. Let us take $\tau \in (0, \e_k)$ and $z \in Y$ with $z\not\equiv 0$ in $Y$. We set
    \[\tilde z =\displaystyle\frac{{\tau} z}{\|z\|}  \; \text{and}\; h_{\tau} = \zeta_k(\tilde z)(w_k- \tilde z).\]
   Then  Lemma \ref{IFT-N} implies that $\tilde z \in \mc N_{\la,\delta}$ and using  \eqref{PS-seq1} with $(w_1,w_2)= h_\tau$ we get
    \[I_{\la,\delta}(h_{\tau}) - I_{\la,\delta}(w_k) \geq -\frac{1}{k} \|(h_{\tau}- w_k)\|. \]
    Now applying the Mean Value theorem we obtain
    \[ (I_{\la,\delta}^\prime(w_k), h_\tau- w_k) + o(\|h_\tau-w_k\|) \geq -\frac{1}{k}\|h_\tau-w_k \|.\]
    Substituting the value of $h_\tau$ in this, we get
    \[(I_{\la,\delta}^\prime(w_k),-\tilde z )+ (\zeta_k(\tilde z)-1 )  (I_{\la,\delta}^\prime(w_k),w_k - \tilde z) \geq -\frac{1}{k}\|h_\tau-w_k \| + o(\|h_\tau-w_k\|). \]
    Then using the fact that $\zeta_k^\prime(\tilde h)(w_k- \tilde h) \in \mc N_{\la,\delta}$, we get
    \begin{equation}\label{PS-seq6}
     -\tau \left( I_{\la,\delta}^\prime(w_k), \frac{z}{\|z\|} \right) + (\zeta_k(\tilde z)-1 ) (I_{\la,\delta}^\prime(w_k)- I_{\la,\delta}^\prime(h_\tau), w_k- \tilde h) \geq -\frac{1}{k} \|h_\tau - w_k\|  + o(\|h_\tau- w_k\|).
     \end{equation}
    Since $\|h_\tau - w_k\|\leq \tau |\zeta_k(\tilde h)|+ |\zeta_k(\tilde h)-1| \|w_k\|$ and
    \[\lim_{\tau \to 0} \frac{|\zeta_k(\tilde h)- 1|}{\tau} \leq \|\zeta_k^\prime(0)\|.\]
    On passing the limit $\tau \to 0$ in \eqref{PS-seq6}, for some constant $M>0$ we get
    \[\left( I_{\la,\delta}^\prime(w_k), \frac{z}{\|z\|} \right) \leq \frac{M}{k} (1+ \|\zeta_k^\prime (0)\|). \]
    This will prove our claim once we are able to show that $\sup\limits_{k} \|\zeta^\prime_k(0)\| < +\infty$. Let $w=(w_1,w_2) \in Y$ then using H\"{o}lder's inequality we get
    \begin{equation}\label{PS-seq7}
    \begin{split}
    \int_{\Om} (\la|u_k|^{q-1}w_1+ \delta |v_k|^{q-1}w_2)\mathrm{d}x \leq (\la+\delta)C_q^q\|(u_k,v_k)\|^{q-1} \|(w_1,w_2)\|,
    \end{split}
    \end{equation}
    where $C_q = \sup\{\int_{\Om}u^q:\; \|u_k\|=1\}$. Again using H\"{o}lder inequality, Hardy-Littlewood-Sobolev inequality and fractional Sobolev embeddings, we can estimate the following
    \begin{equation}\label{PS-seq8}
    \begin{split}
    &\int_\Om (|x|^{\mu}\ast |u_k|^{2^*_\mu}) |v_k|^{2^*_\mu-1}w_1~\mathrm{d}x\\
     & \leq {C(n,\mu)}\left( \int_\Om \left( |v_k|^{2^*_\mu-1}w_1 \right)^{\frac{2n}{2n-\mu} } \right)^{\frac{2n-\mu}{2n}} \left( \int_\Om |u_k|^{2^*_\mu\cdot \frac{2n}{2n-\mu} }\right)^{\frac{2n-\mu}{2n}}\\
    & \leq  {C(n,\mu)}\left[\left( \int_\Om |v_k|^{{2^*_s}}\right)^{\frac{n-\mu+2s}{{2n-\mu}}} \left(\int_\Om |w_1|^{2^*_s} \right)^{\frac{1}{2^*_\mu} } \right]^{\frac{2n-\mu}{2n}} \left( \int_\Om |u_k|^{2^*_s }\right)^{\frac{2n-\mu}{2n}}\\
    & \leq  M_1 \|(u_k,v_k)\|^\alpha \|(w_1,w_2)\|,
    \end{split}
    \end{equation}
where $\alpha = \displaystyle {2^*_s\left(\frac{3n-2\mu+2s}{2n}\right)}$ and $  M_1 >0$ is a constant. Similarly we can show that there exist $ M_2>0$ such that
    \begin{equation}\label{PS-seq9}
    \begin{split}
    \int_\Om (|x|^{\mu}\ast |v_k|^{2^*_\mu}) |u_k|^{2^*_\mu-1}w_2~\mathrm{d}x & \leq M_2 \|(u_k,v_k)\|^{{\alpha}}\|(w_1,w_2)\|.
    \end{split}
    \end{equation}
    Consequently using \eqref{PS-seq7}, \eqref{PS-seq8} and \eqref{PS-seq9} in \eqref{IFT-N1} we get
    \[ |(\zeta^\prime_k(0), w)| \leq \frac{M_3 \|(w_1,w_2)\|}{|(2-q) {C_s^n}\|(u_k,v_k)\|^2-2(22^*_\mu-q){B(u_k,v_k)}  |}\]
    where $M_3>0$  is a constant independent of $(u_k,v_k)$, thanks to {\eqref{PS-seq4}}.\\
    \textbf{Claim:} There exists a $M_4>0$ such that
    \[\left|(2-q){C_s^n}\|(u_k,v_k)\|^2-2(22^*_\mu-q){B(u_k,v_k)}  \right| \geq M_4.\]
    On contrary, let us assume that there exist a subsequence still denoted by $\{(u_k,v_k)\} \subset \mc N_{\la,\delta}$ such that
    \begin{equation}\label{PS-seq10}
     \left|(2-q){C_s^n}\|(u_k,v_k)\|^2-2(22^*_\mu-q){B(u_k,v_k)} \right| = o_k(1).
     \end{equation}
     Since $(u_k,v_k) \in \mc N_{\la,\delta}$, we have
      \begin{equation*}\label{PS-seq11}
      \begin{split}
      {C_s^n}\|(u_k,v_k)\|^2 &= \left( \frac{22^*_\mu-q}{22^*_\mu-2} \right)\Q + o_k(1)\\
      & \leq \left( \frac{22^*_\mu-q}{22^*_\mu-2} \right)  S_s^{-\frac{q}{2}} |\Om|^{\frac{2^*_s -q}{2^*_s}} (\la^{\frac{2}{2-q}} + \delta^{\frac{2}{2-q}})^{\frac{2-q}{2} } \|(u_k,v_k)\|^q +o_k(1)
      \end{split}
      \end{equation*}
      {which implies that
     \begin{equation}\label{PS-seq11}
      \begin{split}
      {C_s^n}\|(u_k,v_k)\|^{2-q}  \leq \left( \frac{22^*_\mu-q}{22^*_\mu-2} \right) S_s^{-\frac{q}{2}} |\Om|^{\frac{2^*_s -q}{2^*_s}} (\la^{\frac{2}{2-q}} + \delta^{\frac{2}{2-q}})^{\frac{2-q}{2} } +o_k(1).
      \end{split}
      \end{equation}
     Also \eqref{PS-seq10} gives us
      \begin{equation*}\label{PS-seq12}
      \begin{split}
       {C_s^n}\|(u_k,v_k)\|^2  = \left( \frac{2(22^*_\mu-q)}{2-q} \right)B(u_k,v_k) + o_k(1)\leq \left( \frac{2(22^*_\mu-q)}{2-q} \right) (\tilde S_s^H)^{- 2^*_\mu}  \|(u_k,v_k)\|^{{2}2^*_\mu} + o_k(1)
       \end{split}
      \end{equation*}
      which implies that
      \begin{equation}\label{PS-seq12}
      \begin{split}
       \|(u_k,v_k)\| \geq \left( \frac{{C_s^n}(2-q)(\tilde S_s^H)^{ 2^*_\mu} }{2(22^*_\mu-q)} \right)^{\frac{1}{22^*_\mu-2}}  + o_k(1)
       \end{split}
      \end{equation}
      where we used the fact that $\|(u_k,v_k)\|\neq o_k(1)$ because of \eqref{PS-seq4}.
From \eqref{PS-seq11} and \eqref{PS-seq12}, for large $k$ we obtain
\begin{align*}
{C_s^n}\left( \frac{{C_s^n}(2-q)(\tilde S_s^H)^{ 2^*_\mu} }{2(22^*_\mu-q)} \right)^{\frac{2-q}{22^*_\mu-2}} \leq \left( \frac{22^*_\mu-q}{22^*_\mu-2} \right) S_s^{-\frac{q}{2}} |\Om|^{\frac{2^*_s -q}{2^*_s}} (\la^{\frac{2}{2-q}} + \delta^{\frac{2}{2-q}})^{\frac{2-q}{2} }
\end{align*}
Then using Lemma \ref{reltn} and \eqref{relation}, the above inequality yields
\[ (\la^{\frac{2}{2-q}} + \delta^{\frac{2}{2-q}}) \geq \left[ \frac{2^{2^*_\mu-1}{(C_s^n)^{\frac{22^*_\mu-q}{2-q}}} }{C(n,\mu)}\left(\frac{2-q}{22^*_\mu-q}\right) \left( \frac{22^*_\mu-2}{22^*_\mu-q}\right)^{\frac{22^*_\mu-2}{2-q}} {S_s}^{\frac{q(2^*_\mu-1)}{2-q}+2^*_\mu}|\Om|^{-\frac{(2^*_s-q)(22^*_\mu-2)}{2^*_s(2-q)}}\right]^{\frac{1}{2^*_\mu-1}} \]}
This contradicts the assumption that $0<  \la^{\frac{2}{2-q}} + \delta^{\frac{2}{2-q}}< \Theta$. Hence the claim holds true and we finally obtain
\[ \left( I_{\la,\delta}^\prime(w_k), \frac{z}{\|z\|} \right) \leq \frac{M}{k}.\]
This establishes our first claim and completes the proof.\hfill{\QED}
\end{proof}

\subsection{First solution}
We now prove the existence of first solution for the problem $(P_{\la,\delta})$.

\begin{Theorem}\label{first-sol}
Let $0<  \la^{\frac{2}{2-q}} + \delta^{\frac{2}{2-q}}< \Theta$. Then there exists a $(u_1,v_1) \in \mc N_{\la,\delta}^+$ such that $(u_1,v_1)$ is a weak solution of $(P_{\la,\delta})$. Moreover, $(u_1,v_1)$ satisfies
 $I_{\la,\delta}(u_1,v_1) = l_{\la,\delta} = l_{\la,\delta}^+ <0$.
\end{Theorem}
\begin{proof}
By Theorem \ref{PS-seq} we know that there exists a $(PS)_{l_{\la,\delta}}$  sequence $\{(u_k,v_k)\}\subset \mc N_{\la,\delta}$ for $I_{\la,\delta}$ that is
 \[\lim_{k \to \infty} I_{\la,\delta}(u_k,v_k) = l_{\la,\delta} \leq l_{\la,\delta}^+ <0\; \text{and}\; \lim_{k \to \infty} I_{\la,\delta}^\prime (u_k,v_k) = 0 \; \text{in}\; Y^*. \]
 By Lemma \ref{PS-seq-bdd} we know that this sequence  $\{(u_k,v_k)\}$ is bounded in $Y$. Therefore there exists $(u_1,v_1) \in Y$ such that upto a subsequence, $(u_k, v_k) \rightharpoonup (u_1,v_1) $ weakly in $Y$ and $(u_k ,v_k) \to (u_1,v_1)$ strongly in $L^m(\Om)$, for $m \in [1,2^*_s)$ as $k \to \infty$. Therefore
$\lim\limits_{k \to \infty}\Q =\q $.
 We already know that $(u_1,v_1)$ is a weak solution of $(P_{\la,\delta})$, by Lemma \ref{unif-lwr-bd}. {Since $\{(u_k,v_k)\} \subset \mc N_{\la,\delta}$ we obtain
 \begin{align*}
 I_{\la,\delta}(u_k,v_k) & = \left(\frac{1}{2}- \frac{1}{22^*_\mu} \right){C_s^n} \|(u_k,v_k)\|^2 - \left( \frac{1}{q}- \frac{1}{22^*_\mu}\right)\Q\\
 & \geq - \left( \frac{1}{q}- \frac{1}{22^*_\mu}\right)\Q.
 \end{align*}
 From Lemma \ref{inf-pos-neg} we know that $l_{\la,\delta}<0$, so passing on the limit $k \to \infty$ we get
 \[\int_\Om(\la|u_1|^q+\delta|v_1|)\mathrm{d}x  \geq - \frac{22^*_\mu}{(22^*_\mu-q)}l_{\la,\delta} >0.\]
 This implies  that $(u_1,v_1) \in \mc N_{\la,\delta}$ is non-trivial solution of $(P_{\la,\delta})$.\\ }
 \textbf{Claim:} $(u_k,v_k) \to (u_1,v_1)$ strongly in $Y$ as $k \to \infty$ and $I_{\la,\delta}(u_1,v_1)= l_{\la,\delta}^+$.\\
Using $(u_1,v_1) \in \mc N_{\la,\delta}$  and Fatou's Lemma we have
\begin{align*}
l_{\la,\delta} \leq I_{\la,\delta}(u_1,v_1) &= \left( \frac{2^*_\mu -1}{22^*_\mu}\right){C_s^n} \|(u_1,v_1)\|^2 - \left( \frac{22^*_\mu-q}{22^*_\mu q}\right) \int_{\Om} (\la |u_1|^q + \delta |v_1|^q)~\mathrm{d}x\\
& \leq \liminf_{k \to \infty} \left( \left( \frac{2^*_\mu -1}{22^*_\mu}\right){C_s^n} \|(u_k,v_k)\|^2 - \left( \frac{22^*_\mu-q}{22^*_\mu q}\right)\Q \right)\\
& = \liminf_{k \to \infty} I_{\la,\delta}(u_k,v_k) = l_{\la,\delta}.
\end{align*}
This implies that $I_{\la,\delta}(u_1,v_1) = l_{\la,\delta}$ and $\|(u_k,v_k)\| \to \|(u_1,v_1)\|$ as $k \to \infty$. We have
\[\|(u_k-u_1, v_k-v_1)\|^2 = \|(u_k,v_k)\|^2 - \|(u_1,v_1)\|^2 +o_k(1). \]
Therefore $(u_k,v_k)\to (u_1,v_1)$ strongly in $Y$ as $k \to \infty$. To establish our claim, it remains to show that  $(u_1,v_1)\in \mc N_{\la,\delta}^+$. On  the contrary, if $(u_1,v_1) \in \mc N_{\la,\delta}^-$ then by Lemma \ref{Theta-def-lem}, there exist unique $t_2 > t_1>0$ such that
\[(t_1u,t_1v) \in \mc N_{\la,\delta}^+ \; \text{and}\; (t_2u_1, t_2v_1) \in \mc N_{\la,\delta}^-.\]
Particularly, $t_1 < t_2 =1$. Since
$\varphi_{u,v}^\prime(t_1)=0 \; \text{and}\; \varphi^{\prime \prime}(t_1)>0$,
so $t_1$ is local minimum of $\varphi_{u,v}$. Therefore there exists a $\hat t \in (t_1,1]$ such that $I_{\la,\delta}(t_1u_1,t_1v_1) < I_{\la,\delta}(\hat tu_1,\hat t v_1) $. Hence
\[l_{\la,\delta} \leq I_{\la,\delta}(t_1u_1,t_1v_1) < I_{\la,\delta}(\hat t u_1,\hat tv_1) \leq I_{\la,\delta}(u_1,v_1) = l_{\la,\delta}\]
which contradicts that $(u_1,v_1) \in \mc N_{\la,\delta}^-$.\hfill{\hfill{\QED}}
\end{proof}

\begin{Lemma}\label{first-sol-pos}
There exists a non negative local minimum of $I_{\la,\delta}$.
\end{Lemma}
\begin{proof}
Suppose $(u_1,v_1)$ be as obtained in Theorem \ref{first-sol}. Then it is also a local minimum for $I_{\la,\delta}$, the proof follows as [pp. $291$,\cite{tarantello}]. If $u_1,v_1 \geq 0$ then we are done. Else consider $(|u_1|,|v_1|)$ then by Lemma \ref{Theta-def-lem} we know that there exist a $t_1$ such that $(t_1|u_1|,t_1|v_1|) \in \mc N_{\la,\delta}^+$. Since $m_{|u_1|,|v_1|}(1) \leq m_{u_1,v_1}(1)= 2B(u_1,v_1)= 2B(|u_1|,|v_1|)=m_{|u_1|,|v_1|}(t_1) $  and $0< m_{u_1,v_1}^\prime(1) \leq m_{|u_1|,|v_1|}^\prime(1)$. This implies $t_1\geq 1$ and thus we have
\[I_{\la,\delta}(t_1|u_1|,t_1|v_1|) \leq I_{\la,\delta}(|u_1|,|v_1|)\leq I_{\la,\delta}(u_1,v_1)= \inf I_{\la,\delta}(\mc N_{\la,\delta}^+).  \]
Hence we obtain a non negative local minimum of $I_{\la,\delta}$ over $\mc N_{\la,\delta}^+$. \hfill{\QED}
\end{proof}\\

We prove positivity of the solution $(u_1,v_1)$ of $(P_{\la,\delta})$.

\begin{Proposition}\label{positive-sol}
The non negative weak solution $(u_1,v_1)$ of $(P_{\la,\delta})$ obtained in Lemma \ref{first-sol-pos} is positive in $\Om$ that is $u_1,v_1>0$ in $\Om$. Moreover for each compact subset $K$ of $\Om$, there exists a $m_K>0$ such that $u_1,v_1 \geq m_K$ in $K$.
\end{Proposition}
\begin{proof}
 We divide the proof into two cases. Consider $u_1$ first and $v_1$ can be shown to be positive in exactly same way.\\
 \textbf{Case(1):} Let $\displaystyle\frac{2^*_s}{(q-1)}> \frac{n}{2s}$ then there exists a sequence $\{u_\e\}_{\e>0}\subset C_c^\infty(\Om)$ such that $u_\e \to u_1$ in $L^{2^*_s}(\Om)$ as $\e \to 0$. That means $u_\e^{q-1} \to u_1^{q-1}$ in $L^{\frac{2^*_s}{(q-1)}}(\Om)$ as $\e\to0$. Now let
\begin{equation*}\label{app1}
w_{\e}:= (-\De)^{-s} (\la u_{\e}^{q-1}).
\end{equation*}
Then using Proposition $1.4(iii)$ of \cite{rs1}, we get that $\{w_\e\}$ is a Cauchy sequence in $C^\beta(\mb R^n)$ where $\beta = \min \{s,2s-\frac{n}{p}\}$ and
\begin{equation}\label{first-sol-pos1}
\|w_{\e}\|_{C^\beta(\mb R^n)} \leq C\|u_{\e}^{q-1}\|_{L^{\frac{2^*_s}{(q-1)}}(\Om)}.
\end{equation}
{We know that there exists a $h \in L^{\frac{2^*_s}{q-1}}(\Om)$ such that $w_{\e} \leq h$, so by Lebesgue Dominated convergence theorem we get
\[\limsup_{\e>0} \int_{\mb R^n} ((-\De)^sw_{\e})w_{\e}~\mathrm{d}x < +\infty.\]
This implies that $\{w_\e\}$ is bounded in $X_0$, hence up to a subsequence, $w_{\e}$ converges weakly to a $w\in X_0$ in $X_0$ as $\e \to 0$.
}Then $w$ satisfies the equation
\[(-\De)^s w = \la u_1^{q-1} \; \text{in}\; \Om,\; w=0\; \text{in}\; \mb R^n \setminus \Om \]
then $w_{\e} \to w$ in $C^{\beta}(\mb R^n)$ so passing on the limit as $ \e \to 0$ in \eqref{first-sol-pos1} we obtain $w \in C(\bar{\Om})$. Since $(u_1,v_1)$ solves $(P_{\la,\delta})$ it is clear that $u_1$ satisfies
\[(-\De)^s u_1 \geq \la u_1^{q-1}\; \text{in}\; \Om, \; u_1 =0 \; \text{in}\; \mb R^n \setminus \Om.\]
Therefore $u_1 \geq w$ in $\Om$, thanks to comparison principle (refer Proposition $4.1$ in \cite{rs2}). Also now by strong maximum principle (refer \cite{Silvestre}), we conclude that $w>0$ in $\Om$ and there exists a $m_K>0$ for each $K$ compact subset of $\Om$ such that $w> m_K$ in $K$.\\
\textbf{Case(2):} Let $\displaystyle\frac{2^*_s}{(q-1)}\leq \frac{n}{2s}$ and consider the following iterative scheme
\[(-\De)^s w_k = \la w_{k-1}^{q-1} \; \text{in}\; \Om,\; w_k = 0\; \text{in}\; \mb R^n \setminus \Om\]
with $w_0 = u_1$. Then take $k=1$ at first and let $\{w_{0,\e}\} \subset C_c^\infty(\Om)$ such that $w_{0,\e} \to w_0=u_1$ in $L^{2^*_s}(\Om)$ as $\e \to 0$ which means $w_{0,\e}^{q-1} \to u_1^{q-1}$ in $L^{\frac{2^*_s}{q-1}}(\Om)$ as $\e \to 0$. We define
\begin{equation*}\label{app2}
w_{\e}^1:= (-\De)^{-s} (\la w_{0,\e}^{q-1}).
\end{equation*}
Set $ q_1 = \textstyle \frac{2^*_s}{q-1}$ and we get using Proposition $1.4(ii)$ of \cite{rs1} that $\{w_{\e}^1\}$ is a Cauchy sequence in $L^{q_2}(\Om)$ where $q_2 =\textstyle \frac{nq_1}{n-2q_1s}> q_1$ and
\begin{equation}\label{first-sol-pos2}
\|w_{\e}^1\|_{L^{q_2}(\Om)} \leq C\|w_{0,\e}^{q-1}\|_{L^{q_1}(\Om)}.
\end{equation}
Necessarily $w_{\e}^1 \to w_1$ as $\e \to 0$ in $L^{q_2}(\Om)$ so passing on the limit as $\e \to 0$ in \eqref{first-sol-pos2} we obtain $w_1 \in L^{q_2}(\Om)$. Proceeding similarly, at each stage we get $w_k \in L^{q_k}(\Om)$ where $q_k = \textstyle \frac{nq_{k-1}}{n-2q_{k-1}s}$ and note that $w_k \not\equiv 0$ for each $k$. Clearly $\{q_k\}$ forms an increasing sequence and the map $t \mapsto \frac{nt}{n-2st} $ has no fixed point. So obviously there exists a $k_0>0$ such that $q_{k_0}> \frac{n}{2s}$ and for this $k_0$ we get $w_{k_0+1} \in C^{\beta}(\mb R^n)$, by Proposition $1.4(iii)$ of \cite{rs1}. By comparison principle we already know that $\{w_k\}$ forms a non increasing sequence and $u_1 \geq w_1$. Thus arguing same as Case(1) we get
\[u_1 \geq w_1 \geq w_2 \geq \ldots \geq w_{k_0+1}>0\; \text{in}\; \Om.\]
Also there exists a $m_K>0$ for each $K$ compact subset of $\Om$ such that $w_{k_0+1}> m_K$ in $K$.\hfill{\QED}
\end{proof}\\

\noi This result suggests that there is no harm to consider $(u_1,v_1)$ as positive (as this property of the first solution will be used further while proving the existence of second solution in the case $\mu \leq 4s$).

\subsection{Second solution}
Now, we establish the existence of second solution for $(P_{\la,\delta})$. We prove this by showing that minimum of $I_{\la,\delta}$ is achieved over $\mc N_{\la,\delta}^-$. {We consider two cases separately that is when $\mu \leq 4s$ and when $\mu \geq 4s$. In the first case we are able to show that when $0<\la^{\frac{2}{2-q}}+ \delta^{\frac{2}{2-q}} < \Theta$, $(P_{\la,\delta})$ has two weak solutions whereas in the other case for $\mu > 4s$ we get another threshold $\Theta_0$ which may be '\textit{less than or equal to}' $\Theta$ such that whenever $0<\la^{\frac{2}{2-q}}+ \delta^{\frac{2}{2-q}} < \Theta_0$, $(P_{\la,\delta})$ possesses two weak solutions. }\\

 \begin{Lemma}\label{l^-<PScrit-lev}
If $\mu \leq 4s$ and $0 < \la^{\frac{2}{2-q}}+ \delta^{\frac{2}{2-q}}< \Theta$, then there exists $(w_0,z_0) \in Y\setminus \{(0,0)\}$ such that $w_0,z_0 \geq 0$ and
 \[ \sup_{t \geq 0} I_{\la,\delta}((u_1,v_1)+ t(w_0,z_0) )< c_1:=I_{\la,\delta}(u_1,v_1)+ \frac{n-\mu+2s}{2n-\mu} \left(\frac{{C_s^n} \tilde S_s^H}{2} \right)^{\frac{2n-\mu}{n-\mu+2s}}. \]
 \end{Lemma}
\begin{proof}
 Using {\eqref{esti-new}}, we can find $r_1>0$ such that
\begin{equation}\label{l^-eq1}
\ds \int_{\mb R^{2n}}\frac{|u_\epsilon(x)- u_\epsilon(y)|^2}{|x-y|^{n+2s}}~\mathrm{d}x\mathrm{d}y {\leq \left((C(n,\mu))^{\frac{n-2s}{2n-\mu}}S_s^H\right)^{\frac{n}{2s}}}+ r_1\epsilon^{n-2s}.
\end{equation}
Also using Proposition \ref{estimates1}, we can find $r_2>0$ such that
\begin{equation}\label{l^-eq2}
\int_{\Om} (|x|^{-\mu} \ast |u_\e|^{2^*_\mu})|u_\e|^{2^*_\mu}~\mathrm{d}x \geq C(n,\mu)^{ {\frac{n}{2s}}} (S_s^H)^{\frac{2n-\mu}{2s}}- r_2 \e^{{2n-\mu}}.
\end{equation}
From proof of Lemma $5.1$ of \cite{TS-ejde}, we know that for fixed $\rho$ such that $1< \rho< \displaystyle\frac{n}{n-2s}$ we have
\begin{equation}\label{l^-eq3}
\int_{\Om}|u_\e|^{\rho} \leq r_3 \e^{\frac{(n-2s)\rho}{2}},
\end{equation}
where $r_3>0$ is an appropriate constant. Now let $0< \e < \delta$ then $u_\e = U_\e$ {in $B_\e(0)$}. \\
\textbf{Claim: } There exists a constant $r_4>0$ such that
\begin{equation}\label{l^-eq4}
\int_{|x|\leq \e}\int_\Om \frac{|u_\e(x)|^{2^*_\mu-1}|u_\e(y)|^{2^*_\mu}}{|x-y|^\mu}~\mathrm{d}x{\mathrm{d}y}  \geq r_4 \e^{\frac{n-2s}{2}}.
\end{equation}
To show this, we split the left hand side of \eqref{l^-eq4} into two integrals and estimate them separately. We recall the definition of $u_\e$ and firstly consider
\begin{equation*}
 \begin{split}
&\int_{|x|\leq \e}\int_{|y|\leq \e}  \frac{|u_\e(x)|^{2^*_{\mu}-1}|u_\e(y)|^{2^*_{\mu}}}{|x-y|^{\mu}}~\mathrm{d}y\mathrm{d}x\\
  &=  \frac{\alpha^{22^*_\mu-1}}{\|\tilde u\|_{L^{2^*_s}(\mb R^n)}^{22^*_\mu-1}} \int_{|x|\leq \e}\int_{|y|\leq \e} \frac{ \e^{\frac{(2s-n)(22^*_\mu-1)}{2}}}{|x-y|^\mu \left(\beta^2+\left|\frac{x}{\e S_s^{\frac{1}{2s}}}\right|^2\right)^\frac{(n-2s)(2^*_\mu-1)}{2}\left(\beta^2 +\left|\frac{y}{\e S_s^{\frac{1}{2s}}}\right|^2\right)^\frac{(n-2s)2^*_\mu}{2}}~\mathrm{d}y\mathrm{d}x\\
  & \geq  E_1 \int_{|x|\leq \e}\int_{|y|\leq \e}\frac{\e^{\frac{(2s-n)(22^*_\mu-1)}{2}-\mu}}{ \left(1+|\frac{x}{\e}|^2\right)^\frac{(n-2s)(2^*_\mu-1)}{2}\left(1+|\frac{y}{\e}|^2\right)^\frac{(n-2s)2^*_\mu}{2}}~\mathrm{d}y\mathrm{d}x\\
  & =  E_1\int_{|x|\leq 1}\int_{|y|\leq 1} \frac{\e^{\frac{n-2s}{2}}}{ (1+|x|^2)^\frac{(n-2s)(2^*_\mu-1)}{2}(1+|y|^2)^\frac{(n-2s)2^*_\mu}{2}}~\mathrm{d}y\mathrm{d}x = {O\left(\e^{\frac{n-2s}{2}}\right)}
 \end{split}
 \end{equation*}
 where $E_1>0$ is appropriate constant that changes value at each step. Secondly, in a similar manner we get
  \begin{equation*}
 \begin{split}
  &\int_{|x|\leq \e}\int_{|y|> \e} \frac{|u_\e(y)|^{2^*_{\mu}}|u_\e(x)|^{2^*_{\mu}-1}}{|x-y|^{\mu}}~\mathrm{d}y\mathrm{d}x\\
   &=  \frac{\alpha^{22^*_\mu-1}}{\|\tilde u\|_{L^{2^*_s}(\mb R^n)}^{22^*_\mu-1}} \int_{|x|\leq \e}\int_{|y|> \e} \frac{ \e^{\frac{(2s-n)(22^*_\mu-1)}{2}}}{|x-y|^\mu \left(\beta^2+\left|\frac{x}{\e S_s^{\frac{1}{2s}}}\right|^2\right)^\frac{(n-2s)(2^*_\mu-1)}{2}\left(\beta^2 +\left|\frac{y}{\e S_s^{\frac{1}{2s}}}\right|^2\right)^\frac{(n-2s)2^*_\mu}{2}}~\mathrm{d}y\mathrm{d}x\\
  & \geq E_1^\prime \int_{|x|\leq \e}\int_{|y|> \e} \frac{\e^{\frac{(2s-n)(22^*_\mu-1)}{2}-\mu}}{ (|y|+\e)^\mu \left(1+|\frac{x}{\e}|^2\right)^\frac{(n-2s)(2^*_\mu-1)}{2}\left(1+|\frac{y}{\e}|^2\right)^\frac{(n-2s)2^*_\mu}{2}}~\mathrm{d}y\mathrm{d}x\\
  & =  E_1^\prime \int_{|x|\leq 1}\int_{|y|> 1} \frac{\e^{\frac{n-2s}{2}}}{ (1+|x|^2)^\frac{(n-2s)(2^*_\mu-1)}{2}(1+|y|^2)^\frac{(n-2s)2^*_\mu}{2}(1+|y|)^\mu}~\mathrm{d}y\mathrm{d}x = {O}\left(\e^{\frac{n-2s}{2}}\right).
 \end{split}
 \end{equation*}
 where $E_1^\prime>0$ is appropriate constant that changes value at each step. This establishes our claim.  We can find appropriate constants $\rho_{1,\la}, \rho_{1,\delta}, \rho_2 >0$ such that the following inequalities holds :\\
(1) $\displaystyle \la \left( \frac{(c+d)^{q}}{q} - \frac{c^{q}}{q} - dc^{1-q} \right) \geq -\frac{\rho_{1,\la}d^{\rho}}{r_3}$ and
 $\displaystyle \delta \left( \frac{(c+d)^{q}}{q} - \frac{c^{q}}{q} - dc^{1-q} \right) \geq -\frac{\rho_{1,\delta}d^{\rho}}{r_3}$ ,
 for all ${c >0}, d\geq 0$.\\
(2) {For each $\e>0$, $m\leq u_1,v_1 $ on compact subsets of $\Om$ where $m>0$ is a constant, we get
\begin{align*}
&\frac{1}{2^*_\mu} B(u_1+tu_\e,v_1+tu_\e)- \frac{1}{2^*_\mu}B(u_1,v_1)- \int_\Om\int_\Om \frac{|u_1(y)|^{2^*_\mu} |v_1(x)|^{2^*_\mu-2}v_1(x)t u_\e(x)}{|x-y|^\mu}~\mathrm{d}y\mathrm{d}x\\
 & \quad \quad- \int_\Om\int_\Om \frac{|v_1(y)|^{2^*_\mu} |u_1(x)|^{2^*_\mu-2}u_1(x)t u_\e(x)}{|x-y|^\mu}~\mathrm{d}y\mathrm{d}x\\
& \geq \frac{t^{22^*_\mu}}{2^*_\mu} B(u_\e,u_\e)+ \frac{\rho_2 t^{22^*_\mu-1}}{ (22^*_\mu-1)} \int_{|x|\leq \e}\int_\Om \frac{|u_\e(y)|^{2^*_{\mu}}|u_\e(x)|^{2^*_{\mu}-1}}{|x-y|^{\mu}}~\mathrm{d}y\mathrm{d}x.
\end{align*}}
{We remark that such an $m$ exists because of Proposition \ref{positive-sol}.}
From Theorem \ref{first-sol} we know that $(u_1,v_1)$ is a weak solution of $(P_{\la,\delta})$. Therefore, we have
\begin{align*}
&I_{\la,\delta}((u_1,v_1)+ t(u_\e,u_\e))- I_{\la,\delta}(u_1,v_1)\\
& = I_{\la,\delta}((u_1,v_1)+ t(u_\e,u_\e))- I_{\la,\delta}(u_1,v_1) - t\bigg( \langle u_1,u_\e \rangle + \langle v_1,u_\e \rangle \bigg.\\
& \quad \left.-  \int_\Om (\la|u_1|^{q-2}u_1u_\e + \delta |v_1|^{q-2}v_1u_\e)~\mathrm{d}x- \int_\Om (|x|^{-\mu}\ast |u_1|^{2^*_\mu})|v_1|^{2^*_\mu-2}v_1u_\e~\mathrm{d}x\right.\\
& \quad \quad\left. - \int_\Om (|x|^{-\mu}\ast |v_1|^{2^*_\mu})|u_1|^{2^*_\mu-2}u_1u_\e~\mathrm{d}x  \right)\\
&= \frac{t^2}{2} {C_s^n} \|(u_\e,u_\e)\|^2 - \la \int_\Om \left(\frac{|u_1+tu_\e|^q-|u_1|^q}{q}- t|u_1|^{q-2}u_1u_\e \right)\mathrm{d}x\\
& \quad - \delta \int_\Om \left(\frac{|v_1+tu_\e|^q-|v_1|^q}{q}- t|{v_1}|^{q-2}v_1u_\e \right)\mathrm{d}x - \left( \frac{B(u_1+tu_\e,v_1+tu_\e)-B(u_1,v_1) }{2^*_\mu}\right.\\
& \quad \quad \left.- \int_\Om (|x|^{-\mu}\ast |u_1|^{2^*_\mu})|v_1|^{2^*_\mu-2}v_1 {t}u_\e~\mathrm{d}x - \int_\Om (|x|^{-\mu}\ast |v_1|^{2^*_\mu})|u_1|^{2^*_\mu-2}u_1{t}u_\e~\mathrm{d}x  \right)
\end{align*}
{which on  using inequality $(2)$ with \eqref{l^-eq1}-\eqref{l^-eq4} gives}
\begin{align*}
&I_{\la,\delta}((u_1,v_1)+ t(u_\e,u_\e))- I_{\la,\delta}(u_1,v_1)\\
&\leq {t^2} {C_s^n} \left( C(n,\mu)^{\frac{n(n-2s)}{2s(2n-\mu)}}(S_s^H)^{\frac{n}{2s}}+ r_1 \e^{n-2s}\right) + (\rho_{1,\la}+ \rho_{1,\delta})t^\rho\e^{\frac{(n-2s)\rho}{2}}\\
&\quad - \frac{t^{{2}2^*_\mu}}{{2^*_\mu}}\left(C(n,\mu)^{\frac{n}{2s}}(S_s^H)^{\frac{2n-\mu}{2s}}- r_2\e^{{2n-\mu}} \right)
- \frac{t^{22^*_\mu-1}\rho_2}{(22^*_\mu-1)}r_4\e^{\frac{n-2s}{2}}.
\end{align*}
{Now we define the function $h_\e: [0,\infty) \to \mb R$ as
 \begin{align*}
 h_\e(t)& ={t^2} {C_s^n} \left( C(n,\mu)^{\frac{n(n-2s)}{2s(2n-\mu)}}(S_s^H)^{\frac{n}{2s}}+ r_1 \e^{n-2s}\right) + (\rho_{1,\la}+ \rho_{1,\delta})t^\rho\e^{\frac{(n-2s)\rho}{2}}\\
 &\quad- \frac{t^{{2}2^*_\mu}}{{2^*_\mu}}\left(C(n,\mu)^{\frac{n}{2s}}(S_s^H)^{\frac{2n-\mu}{2s}}- r_2\e^{{2n-\mu}} \right)
- \frac{t^{22^*_\mu-1}\rho_2}{(22^*_\mu-1)}r_4\e^{\frac{n-2s}{2}}.
 \end{align*}
 Then $h_\e$ attains its maximum at
  \begin{align*}
  t_\e = &{(C_s^n)^{\frac{n-2s}{2(n-\mu+2s)}}}C(n,\mu)^{-\frac{n(n-2s)}{4s(2n-\mu)}}(S_s^H)^{-\frac{(n-2s)(n-\mu)}{4s(n-\mu+2s)}}\\
  &\quad - \frac{\rho_2r_4(n-2s)}{4(n-\mu+2s)}C(n,\mu)^{-\frac{n}{2s}}(S_s^H)^{\frac{\mu-2n}{2s}}\e^{\frac{n-2s}{2}}+o(\e^{\frac{n-2s}{2}}).
  \end{align*}
  Therefore we get
\begin{equation*}
\begin{split}
& {\sup_{t \geq 0} (I_{\la,\delta}((u_1,v_1)+ t(u_\e,u_\e))- I_{\la,\delta}(u_1,v_1))} \\
& \leq \frac{n-\mu+2s}{2n-\mu}({C_s^n} S_s^H)^{\frac{2n-\mu}{n-\mu+2s}} - \frac{\rho_2r_4\e^{\frac{n-2s}{2}}}{22^*_\mu-1}C(n,\mu)^{-\frac{n(3n-2\mu+2s)}{4s(2n-\mu)}}(S_s^H)^{\frac{(\mu-n)(3n-2\mu+2s)}{4s(n-\mu+2s)}}+o(\e^{\frac{n-2s}{2}})\\
 &<\frac{n-\mu+2s}{(2n-\mu)} ({C_s^n} S_s^H)^{\frac{2n-\mu}{n-\mu+2s}}= \frac{n-\mu+2s}{(2n-\mu)} \left(\frac{C_s^n\tilde S_s^H}{2}\right)^{\frac{2n-\mu}{n-\mu+2s}}  .
 \end{split}
 \end{equation*}}
Choosing $(w_0,z_0)= (u_\e,u_\e)$, for appropriate choice of $\e$ as shown above, we obtain the result.\hfill{\QED}
\end{proof}


\begin{Corollary}
It holds that $l_{\la,\delta}^- < c_1$.
\end{Corollary}
\begin{proof}
For each $(u,v)\in Y$, by Lemma \ref{Theta-def-lem} we know that there exists a $t_2(u,v)>0$ (notation changed to show that $t_2$ depends on $(u,v)$) such that $t_2( u, v)(u,v)\in \mc N_{\la,\delta}^-$. We consider two sets
\begin{align*}
U_1 &:= \left\{(u,v)\in Y: \; \|(u,v)\|< t_2\left(\frac{(u,v)}{\|(u,v)\|}\right) \right\}\; \text{and}\\
 U_2 &:= \left\{(u,v)\in Y: \; \|(u,v)\|> t_2\left(\frac{(u,v)}{\|(u,v)\|}\right) \right\}.
\end{align*}
\textbf{Claim:} $Y \setminus \mc N_{\la,\delta}^- = U_1 \cup U_2$.\\
For any $(u,v)\in Y$ we define $(\hat u,\hat v) := \textstyle\frac{(u,v)}{\|(u,v)\|}$. Now let $(u,v) \in \mc N_{\la,\delta}^-$. Then we know that there exists a $t_2(\hat u,\hat v)>0$ such that $t_2(\hat u,\hat v)(\hat u,\hat v) \in \mc N_{\la,\delta}^-$. But $(u,v) \in \mc N_{\la,\delta}^-$ implies that it must be that $\textstyle\frac{t_2(\hat u,\hat v)}{\|(u,v)\|} =1 $ which means $t_2(\hat u,\hat v) = \|(u,v)\|$. On the other hand, let $(u,v) \in Y$ be such that $t_2(\hat u,\hat v)=\|(u,v)\|$. By definition $t_2(\hat u,\hat v)(\hat u,\hat v) \in \mc N_{\la,\delta}^-$ which implies that $(u,v)\in \mc N_{\la,\delta}^-$. This proves the claim. \\
Next let $(u,v)\in \mc N_{\la,\delta}^+$ then by Lemma \ref{Theta-def-lem} we know that there exists a $t_1(\hat u,\hat v)>0$ such that $t_1(\hat u,\hat v)(\hat u,\hat v) \in \mc N_{\la,\delta}^+$. But $(u,v)\in \mc N_{\la,\delta}^+$ implies that $\frac{t_1(\hat u,\hat v)}{\|(u,v)\|}=1$. This gives $t_2(\hat u,\hat v)>t_1(\hat u,\hat v)= \|(u,v)\|$ that is $(u,v)\in U_1$. Therefore $\mc N_{\la,\delta}^+ \subset U_1$ and thus $(u_1,v_1)\in U_1$. \\
We consider the map $\gamma_M \in C([0,1],Y)$ defined by $\gamma_M(t):= (u_1,v_1)+ t M(w_0,z_0)$ for $M>0$, where $(w_0,z_0)$ is defined Lemma \ref{l^-<PScrit-lev}. Clearly $\gamma(0)=(u_1,v_1)$ and $\gamma(1)= (u_1,v_1)+M(w_0,z_0)$. There exists a $R>0$ such that $0< t_2(u,v)<R$ on the set $\{(u,v)\in Y:\; \|(u,v)\|=1 \}$. Let us choose $M_0>0$ such that
\[M_0 \geq \frac{|R^2 - \|(u_0,v_0)\|^2 |}{\|(w_0,z_0)\|^2}.\]
Then
\begin{align*}
\|(u_1,v_1)+ M_0(w_0,z_0)\|^2 &\geq \|(u_1,v_1)\|^2 + M_0^2 \|(w_0,z_0)\|^2+ O(M_0) \\
&\geq R^2 > \left(t_2\left(\frac{(u_1,v_1)+ M_0(w_0,z_0)}{\|(u_1,v_1)+ M_0(w_0,z_0)\|}\right)\right)^2
\end{align*}
which implies $(u_1,v_1)+ M_0(w_0,z_0) \in U_2$. Now since $\gamma_{M_0}$ is a continuous path starting from $(u_1,v_1)$ to $(u_1,v_1)+ M_0(w_0,z_0)$ and $Y \setminus \mc N_{\la,\delta}^- = U_1 \cup U_2$, there must exists a $\hat t>0$ such that $\|(u_1,v_1)+ M_0(w_0,z_0)\| = \textstyle t_2\left(\frac{(u_1,v_1)+ M_0(w_0,z_0)}{\|(u_1,v_1)+ M_0(w_0,z_0)\|}\right) $ that is $\gamma_{M_0}(\hat t) \in \mc N_{\la,\delta}^-$. Therefore $(u_1,v_1)+\hat t M_0(w_0,z_0) \in \mc N_{\la,\delta}^-$. Finally using Lemma \ref{l^-<PScrit-lev} we obtain
\[l_{\la,\delta}^- \leq I_{\la,\delta}((u_1,v_1)+ \hat t M_0(w_0,z_0)) \leq \sup_{t \geq 0} I_{\la,\delta}((u_1,v_1)+ t(w_0,z_0))< c_1. \]
This completes the proof. \hfill{\QED
}
\end{proof}

 \begin{Lemma}\label{l^-<PScrit-lev-new}
 If $\mu >4s$ then there exists a $\Upsilon >0$ such that whenever $0 < \la^{\frac{2}{2-q}}+ \delta^{\frac{2}{2-q}}< \Upsilon$, we have
 \[l_{\la,\delta}^- < c_0:= \frac{n-\mu+2s}{(2n-\mu)}\left( \frac{{C_s^n}\tilde{S_s^H}}{2}\right)^{\frac{2n-\mu}{n-\mu+2s} }- D_0\left(\la^{\frac{2}{2-q}}+ \delta^{\frac{2}{2-q} }\right)\]
 where $D_0$ has been defined in Lemma \ref{unif-lwr-bd}.
 \end{Lemma}
\begin{proof}
Let $w_0=z_0 = u_\e$ and define
\[J_{\la,\delta}(u,v)= \frac{{C_s^n}}{2} \|(u,v)\|^2 - \frac{{1} }{2^*_\mu} B(u,v)\; \text{and}\; f(t)= J_{\la,\delta}(tw_0,tz_0).\]
Then $f(0)= 0$ and $f(t)<0$ if $t \in (0, T)$, $f(t)>0$ if $t >T$ where
$T= \textstyle\left( \frac{2^*_\mu {C_s^n}\|(w_0,z_0)\|^2}{2 B(w_0,z_0)} \right)^{\frac{1}{2(2^*_\mu-1)}}.$
It is next easy thing to verify that $f$ attains its maximum at
$t_* = \textstyle\left( \frac{{C_s^n}\|w_0,z_0\|^2}{2 B(w_0,z_0)} \right)^{\frac{1}{2(2^*_\mu-1)}}.$
Therefore using \eqref{esti-new} and Proposition \ref{estimates1} we have
\begin{equation}\label{sec-sol1}
\begin{split}
&\sup_{t \geq 0} J_{\la,\delta}(tw_0,tz_0) \\
&= f(t_*) = \frac{{C_s^n}t_*^2}{2} \|(w_0,z_0)\|^2 - \frac{t_*^{22^*_\mu}}{2^*_\mu} B(w_0,z_0)=  \left( \frac{n-\mu+2s}{2n-\mu} \right) \left( \frac{{C_s^n}\|{u_\e} \|^2}{ B({u_\e,u_\e})^{\frac{1}{2^*_\mu}}} \right)^{\frac{2^*_\mu}{2^*_\mu-1}}\\
& \leq  \left( \frac{n-\mu+2s}{2n-\mu} \right) \left(  \frac{{C_s^n}(C(n,\mu)^{\frac{n-2s}{2n-\mu}\cdot \frac{n}{2s}} (S_s^H)^{\frac{n}{2s}} + {O}(\e^{n-2s}))}{ \left( C(n,\mu)^{\frac{n}{2s}} (S_s^H)^{\frac{2n-\mu}{2s}} - {O}(\e^{2n-\mu}) \right)^{\frac{n-2s}{2n-\mu}}} \right)^{\frac{2^*_\mu}{2^*_\mu -1}}\\
& \leq \left( \frac{n-\mu+2s}{2n-\mu} \right){(C_s^n)^{\frac{2^*_\mu}{2^*_\mu-1}}} \left( S_s^H + o(\e^{n-2s}) \right)^{\frac{2^*_\mu}{2^*_\mu -1}}\\
 &= \left( \frac{n-\mu+2s}{2n-\mu} \right) \left( \left( \frac{{C_s^n}\tilde S_s^H}{2}\right)^{\frac{2^*_\mu}{2^*_\mu-1}} +o(\e^{n-2s})\right)
\end{split}
\end{equation}
Recalling the definition of $c_0$, we note that if $0 < \la^{\frac{2}{2-q}}+ \delta^{\frac{2}{2-q}}< \Upsilon_1$ where $\Upsilon_1>0$ is chosen such that $c_0 >0$ for example $\Upsilon_1 = \displaystyle\frac{n-\mu+2s}{2 D_0 (2n-\mu)} \left( \frac{{C_s^n}\tilde S_s^H}{2}\right)^{\frac{2n-\mu}{n-\mu+2s}}$. Since $I_{\la,\delta}(tw_0,tz_0) \leq \frac{t^2}{2}\|(w_0,z_0)\|^2$ for $t\geq 0$, we can find $\bar t >0$ such that
$\sup\limits_{t \in [0,\bar t]} I_{\la,\delta}(tw_0,tz_0) < c_0$
whenever $0 < \la^{\frac{2}{2-q}}+ \delta^{\frac{2}{2-q}}< \Upsilon_1$. Let us define function  $H_{\la,\delta} : Y \to \mb R$ as $ H_{\la,\delta}(u,v) := \displaystyle \frac{1}{q} \q.$  Now using \eqref{sec-sol1} we have
\begin{align*}
\sup_{t \geq \bar t} I_{\la,\delta}(tw_0,tz_0) &= \sup_{t \geq \bar t} (J_{\la,\delta}(tw_0,tz_0)- H_{\la,\delta}(tw_0,tz_0))\\\
& \leq  \left( \frac{n-\mu+2s}{2n-\mu} \right) \left( \left( \frac{{C_s^n}\tilde S_s^H}{2}\right)^{\frac{2^*_\mu}{2^*_\mu-1}} + O(\e^{n-2s})\right)- \frac{\bar t^q}{q}(\la+ \delta)\int_{\mb R^n} |u_\e|^q ~\mathrm{d}x\\
& \leq \left( \frac{n-\mu+2s}{2n-\mu} \right)  \left( \frac{{C_s^n}\tilde S_s^H}{2}\right)^{\frac{2^*_\mu}{2^*_\mu-1}} + O(\e^{n-2s})- \frac{\bar t^q}{q}(\la+ \delta)\int_0^{\delta_*} |u_\e|^q ~\mathrm{d}x
\end{align*}
for any $\delta_*>0$. {Fix $\delta_*<\delta$ and letting $0 < \e  < \delta_*$ we estimate
\begin{align*}
\int_{B(0,\delta_*)} |u_\e|^q ~\mathrm{d}x &= \int_{B(0,\delta_*)} |U_\e|^q ~\mathrm{d}x \geq C_1 |S_{n-1}|\e^{n-\frac{(n-2s)q}{2}}\int_0^{\frac{\delta_*}{\e}} \frac{r^{n-1}}{(1 + r^2)^{\frac{(n-2s)q}{2}}}~\mathrm{d}r\\
&\geq C_2 |S_{n-1}|\e^{n-\frac{(n-2s)q}{2}}\int_0^{\frac{\delta_*}{\e}}{r^{n-1-{(n-2s)q}}}~\mathrm{d}r\\
&\geq C_2 |S_{n-1}|\e^{n-\frac{(n-2s)q}{2}}
\begin{cases}
\int_1^{\frac{\delta_*}{\e}}{r^{n-1-{(n-2s)q}}}~\mathrm{d}r\quad \text{if} \; n\leq (n-2s)q\\
\int_0^{\frac{\delta_*}{\e}}{r^{n-1-{(n-2s)q}}}~\mathrm{d}r\quad \text{if} \; n> (n-2s)q
\end{cases}\\
& \simeq C_3
\begin{cases}
&\e^{n-\frac{(n-2s)q}{2}}, \;\text{if} \; n< (n-2s)q\\
& \e^{\frac{n}{2}}|\log\e|, \; \text{if}\; n= (n-2s)q\\
& \e^{\frac{(n-2s)q}{2}}, \; \text{if}\; n> (n-2s)q
\end{cases}
\end{align*}
where $C_1,C_2$ and $C_3$ are appropriate positive constants. Therefore using $1<q<2$ we obtain
 \begin{align*}
&\sup_{t \geq t_0} I_{\la,\delta}(tw_0,tz_0)\\
 &\leq  \left( \frac{n-\mu+2s}{2n-\mu} \right) \left( \frac{\tilde S_s^H}{2}\right)^{\frac{2^*_\mu}{2^*_\mu-1}}+
 \begin{cases}
&O(\e^{n-2s})- (\la+\delta)O(\e^{n-\frac{(n-2s)q}{2}}), \;\; \text{if} \; n< (n-2s)q\\
&-(\la+\delta) O( \e^{\frac{n}{2}}|\log\e|), \;\; \text{if}\;  n= (n-2s)q\\
&- (\la+\delta)O(\e^{\frac{(n-2s)q}{2}}), \;\; \text{if}\; n> (n-2s)q
\end{cases}
\end{align*}
 This implies for $\e = (\la^{\frac{2}{2-q}} + \delta^{\frac{2}{2-q}})^{\frac{1}{n-2s} } \leq \delta_*$
\begin{align*}
\sup_{t \geq t_0} I_{\la,\delta}(tw_0,tz_0) &\leq  \left( \frac{n-\mu+2s}{2n-\mu} \right) \left( \frac{{C_s^n}\tilde S_s^H}{2}\right)^{\frac{2^*_\mu}{2^*_\mu-1}}\\
& +
\begin{cases}
&C( \la^{\frac{2}{2-q}} + \delta^{\frac{2}{2-q}}) - C(\la+ \delta)(\la^{\frac{2}{2-q}} + \delta^{\frac{2}{2-q}})^{\frac{1}{n-2s}({n-\frac{(n-2s)q}{2}})}, \;\text{if} \; n< (n-2s)q\\
& -C(\la+\delta)(\la^{\frac{2}{2-q}} + \delta^{\frac{2}{2-q}})^{\frac{n}{2(n-2s)} }|\log((\la^{\frac{2}{2-q}} + \delta^{\frac{2}{2-q}})^{\frac{1}{n-2s} })|, \; \text{if}\; n= (n-2s)q\\
& -C(\la+\delta)(\la^{\frac{2}{2-q}} + \delta^{\frac{2}{2-q}})^{\frac{q}{2} }, \; \text{if}\; n> (n-2s)q
\end{cases}.
\end{align*}
Let $n< (n-2s)q$ then $\textstyle 1+ \frac{2}{(2-q)(n-2s)}\left(n- \frac{(n-2s)q}{2}\right) < \frac{2}{2-q}$ which implies that we can choose a $\Upsilon_2 >0$ small enough such that if $0<\la^{\frac{2}{2-q}} + \delta^{\frac{2}{2-q}} < \Upsilon_2$ then
\[C\left( \la^{\frac{2}{2-q}} + \delta^{\frac{2}{2-q}}\right)- C(\la+ \delta) (\la^{\frac{2}{2-q}} + \delta^{\frac{2}{2-q}})^{\frac{1}{n-2s}({n-\frac{(n-2s)q}{2}})} < -  D_0 \left( \la^{\frac{2}{2-q}} + \delta^{\frac{2}{2-q}}\right).\]
As $\la,\delta \to 0$, $|\log((\la^{\frac{2}{2-q}} + \delta^{\frac{2}{2-q}})^{\frac{1}{n-2s} })| \to \infty$ so in case $n = (n-2s)q$ we can obtain a $\Upsilon_2>0$ small enough such that
\[-C(\la+\delta)(\la^{\frac{2}{2-q}} + \delta^{\frac{2}{2-q}})^{\frac{n}{2(n-2s)} }|\log((\la^{\frac{2}{2-q}} + \delta^{\frac{2}{2-q}})^{\frac{1}{n-2s} })|
< - D_0 \left( \la^{\frac{2}{2-q}} + \delta^{\frac{2}{2-q}}\right). \]
Else if $n > (n-2s)q$ then $(\la+\delta)(\la^{\frac{2}{2-q}} + \delta^{\frac{2}{2-q}})^{\frac{q}{2} } \simeq (\la^{\frac{2}{2-q}} + \delta^{\frac{2}{2-q}})$ and hence clearly we can obtain  a $\Upsilon_2>0$ small enough such that
\[-C(\la+\delta)(\la^{\frac{2}{2-q}} + \delta^{\frac{2}{2-q}})^{\frac{q}{2} }
< -  D_0 \left( \la^{\frac{2}{2-q}} + \delta^{\frac{2}{2-q}}\right). \]
Setting $\Upsilon = \min\{\Upsilon_1, \Upsilon_2, \delta_*^{n-2s}\} >0$ we finally get that
\[\sup_{t \geq 0} I_{\la,\delta}(tw_0,tz_0) < c_0 \]
whenever $0 < \la^{\frac{2}{2-q}} + \delta^{\frac{2}{2-q}}< \Upsilon$. To prove the last part of the Lemma, we note that there exists $t_2 >0$ such that $(t_2w_0, t_2z_0) \in \mc N^-_{\la,\delta}$ and
\[l_{\la,\delta}^- \leq \; I_{\la,\delta}(t_2w_0, t_2z_0) \leq \; \sup_{t\geq 0} I_{\la,\delta}(tw_0, tz_0)  < c_0 \]
when $0 <\la^{\frac{2}{2-q}}+ \delta^{\frac{2}{2-q}}< \Upsilon$. This concludes the proof.} \hfill{\QED}
\end{proof}

\noi{ Before proving the existence of second solution, we make a remark at this stage.
\begin{Remark}\label{c_0<c_1}
Using Lemma \ref{unif-lwr-bd} it is easy to see that $c_1> c_0$, where $c_1$ is defined in Lemma \ref{l^-<PScrit-lev} and $c_0$ is defined in Lemma \ref{l^-<PScrit-lev-new}.
\end{Remark}}

{\begin{Theorem}\label{sec-sol}
There exists a $(u_2,v_2) \in \mc N_{\la,\delta}^-$ such that
 $I_{\la,\delta}(u_2,v_2) = l_{\la,\delta}^-$ in each of the following cases:
 \begin{enumerate}
 \item[(i)] $0<  \la^{\frac{2}{2-q}} + \delta^{\frac{2}{2-q}}< \Theta $ when $\mu \leq 4s$ and
 \item[(ii)] $0<  \la^{\frac{2}{2-q}} + \delta^{\frac{2}{2-q}}< \Theta_0:=\min\{\Theta, \Upsilon\}$ when $\mu>4s$.
 \end{enumerate}
 Moreover, $(u_2,v_2)$ is a weak solution of $(P_{\la,\delta})$.
\end{Theorem}}
\begin{proof}
Let $\{(u_k,v_k)\} \subset \mc N_{\la,\delta}^-$ be a minimizing sequence such that
$\lim\limits_{k \to \infty}I_{\la,\delta}(u_k,v_k)= l_{\la,\delta}^-$.
By Lemma \ref{inf-pos-neg}$(ii)$, we know that $\{(u_k,v_k)\}$ is a bounded sequence in $X_0$. Hence there exists a $(u_2,v_2) \in Y$ such that, {upto a subsequence}, $(u_k,v_k) \rightharpoonup (u_2,v_2)$ weakly in $X_0$ as $k \to \infty$. \\
\textbf{Claim(1): } As $k \to \infty$, $u_k \to u_2$ and $v_k \to v_2$ strongly in $X_0$.\\
If not, we define $z_k = u_k-u_2$ and $w_k = v_k-v_2$ and assume that as $k \to \infty$
\[\|(z_k,w_k)\|^2 \to c^2 \; \text{and}\; B(z_k,w_k) \to d^{22^*_\mu}. \]
for some $c,d \neq 0$. Then as $k \to \infty$ we have
\[\|(u_k,v_k)\|^2 = \|(z_k,w_k)\|^2+ \|(u_2,v_2)\|^2 + o_k(1). \]
 Before proving claim (1) we state and prove the following.\\
\textbf{Claim(2): } As $k \to \infty$, $B(u_k,v_k) - B(z_k,w_k) \to B(u_2,v_2)$.\\
From fractional Sobolev embedding we have that
\[|z_k|^{2^*_\mu} -|u_k|^{2^*_\mu} \rightharpoonup |u_2|^{2^*_\mu} \; \text{and}\; |w_k|^{2^*_\mu} -|v_k|^{2^*_\mu} \rightharpoonup |v_2|^{2^*_\mu}\; \text{in} \; L^{\frac{2n}{2n-\mu}}(\mb R^n).\]
By Proposition \ref{HLS}, we have
\[|x|^{-\mu}\ast(|z_k|^{2^*_\mu} - |u_k|^{2^*_\mu}) \rightharpoonup |x|^{-\mu}\ast |u_2|^{2^*_\mu} \; \text{and}\; |x|^{-\mu}\ast(|w_k|^{2^*_\mu} - |v_k|^{2^*_\mu}) \rightharpoonup |x|^{-\mu}\ast |v_2|^{2^*_\mu} \; \text{in}\; L^{\frac{2n}{\mu}}(\mb R^n).  \]
Also from boundedness of $\{u_k\}$ and $\{v_k\}$ in $L^{\frac{2n}{n-2s}}(\mb R^n)$ we know that $|z_k|^{2^*_\mu} \rightharpoonup 0$ and $|w_k|^{2^*_\mu} \rightharpoonup 0$ in $L^{\frac{2n}{2n-\mu}}(\mb R^n)$. This gives $B(u_k-z_k,w_k) \to 0$ and $B(v_k-w_k,z_k) \to 0$ as $k \to \infty$. This altogether proves claim(2) because we can write
\begin{align*}
B(u_k,v_k)- B(z_k,w_k) = B(u_k-z_k,v_k-w_k)+ B(v_k-w_k,z_k)+B(u_k-z_k,w_k).
\end{align*}
Since $\{(u_k,v_k)\} \subset \mc N_{\la,\delta}^-$, $\lim\limits_{ k \to \infty} \varphi^\prime_{u_k,v_k}(1)=0$. This gives
\begin{equation}\label{sec-sol4}
\varphi^\prime_{u_2,v_2}(1)+ {C_s^n} c^2 - 2d^{22^*_\mu} = 0.
\end{equation}
\textbf{Claim(3): } $(u_2,v_2)$ is non-trivial.\\
Suppose not and $u_2\equiv 0 \equiv v_2$. This implies $c \neq 0$ because of Lemma \ref{inf-pos-neg}$(ii)$.  Also using definition of $\tilde S_s^H$ and ${C_s^n}c^2 = 2 d^{22^*_\mu}$ (by \eqref{sec-sol4}), we get
\[\frac{c^2}{2} \geq \left(\frac{{C_s^n}\tilde S_s^H}{2} \right)^{\frac{22^*_\mu}{2(2^*_\mu-1)}}.\]
Therefore
\begin{equation}\label{sec-sol1-new}
\begin{split}
l_{\la,\delta}^- = \lim_{k \to \infty} I_{\la,\delta}(u_k,v_k)
&= I_{\la,\delta}(0,0) + \frac{{C_s^n}c^2}{2}- \frac{2d^{22^*_\mu}}{22^*_\mu}\\
 &=\frac{{C_s^n}c^2}{2} \left( 1- \frac{1}{2^*_\mu} \right) \geq \left( \frac{n-\mu+2s}{2n-\mu} \right)\left( \frac{{C_s^n}\tilde S_s^H}{2} \right)^{\frac{2n-\mu}{n-\mu+2s}}.\hspace{5.0cm}
 \end{split}
\end{equation}
{If $\mu\leq 4s$, then using \eqref{sec-sol1-new} with Lemma \ref{l^-<PScrit-lev}, we have that $I_{\la,\delta}(u_1,v_1)>0$ but this is a contradiction to $I_{\la, \delta}(u_1,v_1)= l_{\la,\delta}^+ <0$ (by Lemma \ref{inf-pos-neg}$(i)$). Otherwise if $\mu>4s$, then using \eqref{sec-sol1-new} with Lemma \ref{l^-<PScrit-lev-new}, we get $-D_0(\la^{\frac{2}{2-q}}+\delta^{\frac{2}{2-q}}) \geq 0$ which is again a contradiction.} This proves claim(3). Since $(u_2,v_2)\in Y\setminus \{(0,0)\}$ and $0 < \la^{\frac{2}{2-q}} + \delta^{\frac{2}{2-q}}< \Theta$ {for both the cases $\mu\leq 4s$ as well as $\mu>4s$}, by Lemma \ref{Theta-def-lem} we know that there exists $t_1,t_2$ such that $0 < t_1<t_2$, $t_1(u_2,v_2)\in \mc N_{\la,\delta}^+$ and $t_2(u_2,v_2)\in \mc N_{\la,\delta}^-$. That is $\varphi_{u_2,v_2}^{\prime}(t_1) =0= \varphi_{u_2,v_2}^\prime(t_2)$. Let us define the following two functions
\[f(t)= \frac{{C_s^n}c^2t^2}{2}-\frac{d^{22^*_\mu}t^{22^*_\mu}}{2^*_\mu}\; \text{and}\; g(t)= \varphi_{u_2,v_2}(t)+f(t). \]
Then we consider the three cases as below:
\begin{enumerate}
\item[(i)] $t_2<1$,
\item[(ii)] $t_2 \geq 1$ and $d>0$,
\item[(iii)] $t_2 \geq 1$ and $d=0$.
\end{enumerate}
$(i)$ Using \eqref{sec-sol4} we get $g^\prime(1)= \varphi_{u_2,v_2}^\prime(1)+ {C_s^n} c^2-2d^{22^*_\mu} =0 $. {Since $\{(u_k,v_k)\} \subset \mc N_{\la,\delta}^-$, for all $t>0$ we get
\begin{equation}\label{sec-sol7}
\varphi_{u_k,v_k}(t) \leq \varphi_{u_k,v_k}(1)
 \end{equation}
  Since $g(t)= \lim\limits_{k \to \infty}\varphi_{u_k,v_k}(t)$, passing on the limits as $k \to \infty$ in \eqref{sec-sol7} we obtain $g(t)\leq g(1)$, for $t>0$.} Therefore
\[ l_{\la,\delta}^- = \lim_{k \to \infty}\varphi_{u_k,v_k}(1)=g(1)>g(t_2)\geq I_{\la,\delta}(t_2u_2,t_2v_2) +\frac{{t^2_2} }{2} ({C_s^n}c^2-2d^{22^*_\mu})> I_{\la,\delta}(t_2u_2,t_2v_2) \geq l_{\la,\delta}^- \]
which is  a contradiction.\\
$(ii)$ We define $t_* = \left(\frac{{C_s^n}c^2}{2d^{22^*_\mu}}\right)^{\frac{1}{22^*_\mu-2}}$ and then it is easy to compute that $f(t)$ attains its maximum at $t=t_*$. Also we compute and find that
\begin{equation*}\label{sec-sol5}
f(t_*)=  \frac{n-\mu+2s}{2n-\mu} \left( \frac{{C_s^n}c^2}{2d^2}\right)^{\frac{2^*_\mu}{2^*_\mu-1}} \geq  \frac{n-\mu+2s}{2n-\mu} \left( \frac{{C_s^n}\tilde S_s^H}{2}\right)^{\frac{2^*_\mu}{2^*_\mu-1}}.
\end{equation*}
Moreover $f^\prime(t)= t({C_s^n}c^2-2d^{22^*_\mu}t^{22^*_\mu-2}) >0$ if $t \in (0,t_*)$ and $f^\prime(t)\leq0$ if $t\geq t_*$. Moreover $g(1)= \max\limits_{t>0}g(t) \geq g(t_*)$. So if $t_*\leq 1$ then
\begin{equation}\label{sec-sol6}
\begin{split}
l_{\la,\delta}^- = g(1) &\geq g(t_*) = I_{\la,\delta}(t_*u_2,t_*v_2)+ f(t_*) \geq I_{\la,\delta}(t_1u_2,t_1v_2)+ \frac{n-\mu+2s}{2n-\mu} \left( \frac{{C_s^n}\tilde S_s^H}{2}\right)^{\frac{2^*_\mu}{2^*_\mu-1}}\\
&\geq l_{\la,\delta}^+ +\frac{n-\mu+2s}{2n-\mu} \left( \frac{{C_s^n}\tilde S_s^H}{2}\right)^{\frac{2^*_\mu}{2^*_\mu-1}} \geq I_{\la,\delta}(u_1,v_1)+ \frac{n-\mu+2s}{2n-\mu} \left( \frac{{C_s^n}\tilde S_s^H}{2}\right)^{\frac{2^*_\mu}{2^*_\mu-1}}=c_1
\end{split}
\end{equation}
which is a contradiction to Lemma \ref{l^-<PScrit-lev} {in the case $\mu \leq 4s$. Whereas when $\mu>4s$, using Remark \ref{c_0<c_1} and \eqref{sec-sol6} we get that $l_{\la,\delta}^- \geq c_1 >c_0$ which is a contradiction to Lemma \ref{l^-<PScrit-lev-new}}. Therefore we must have $t_* >1$. Since $g^\prime(t)\leq 0$ for $t \geq 1$, whenever $t \in [1,t_*]$ we get $\varphi^\prime_{u_2,v_2}(t) \leq -f^\prime(t) \leq 0$. This gives either $t_*\leq t_1$ or $t_2=1$. If $t_* \leq t_1$ then \eqref{sec-sol6} holds true and we arrive at a contradiction whereas if $t_2=1$ then $(u_2,v_2) \in \mc N_{\la,\delta}^-$ which implies ${C_s^n}c^2=2d^{22^*_\mu}$ (by \eqref{sec-sol4}). This gives
\begin{align*}
l_{\la,\delta}^- &=g(1)= I_{\la,\delta}(u_2,v_2) + d^{22^*_\mu}\left( 1-\frac{1}{2^*_\mu}\right)\geq I_{\la,\delta}(u_2,v_2) +  \frac{n-\mu+2s}{2n-\mu} \left( \frac{{C_s^n}\tilde S_s^H}{2}\right)^{\frac{2^*_\mu}{2^*_\mu-1}}\\
 &\geq I_{\la,\delta}(t_1u_2,t_1v_2) +  \frac{n-\mu+2s}{2n-\mu} \left( \frac{{C_s^n}\tilde S_s^H}{2}\right)^{\frac{2^*_\mu}{2^*_\mu-1}} \geq I_{\la,\delta}(u_1,v_1)+  \frac{n-\mu+2s}{2n-\mu} \left( \frac{{C_s^n}\tilde S_s^H}{2}\right)^{\frac{2^*_\mu}{2^*_\mu-1}}
\end{align*}
which contradicts Lemma \ref{l^-<PScrit-lev} {in the case $\mu \leq 4s$. Whereas when $\mu>4s$, using Remark \ref{c_0<c_1} and \eqref{sec-sol6} we get that $l_{\la,\delta}^- \geq c_1 > c_0$ which is a contradiction to Lemma \ref{l^-<PScrit-lev-new}}.\\
Hence, only possibility is that $(iii)$ holds true that is $t_2\geq 1$ and $d=0$. If $c \neq 0$ then \eqref{sec-sol4} implies $\varphi_{u_2,v_2}^\prime(1)=-c^2<0$ and also $\varphi^{\prime\prime}_{u_2,v_2}(1)<0$ which is a contradiction since $ t_2 \geq 1$. Thus $c=0$ and this proves claim(1). Therefore $I_{\la,\delta}(u_2,v_2)= l_{\la,\delta}^-$ and obviously $(u_2,v_2)\in \mc N_{\la,\delta}^-$. Finally, $(u_2,v_2)$ is a weak solution of $(P_{\la,\delta})$ follows from Lemma \ref{min-is-sol}.\hfill{\QED}
\end{proof}\\

\subsection{Proof of Main Theorem}

 \textbf{Proof of Theorem \ref{MT}: } {By Theorem \ref{first-sol} and \ref{sec-sol} we know that $(P_{\la,\delta})$ has two solutions $(u_1,v_1) \in \mc N_{\la,\delta}^+$ and $(u_2,v_2) \in \mc N_{\la,\delta}^-$ whenever $0<  \la^{\frac{2}{2-q}} + \delta^{\frac{2}{2-q}}< \Theta$ if $\mu \leq 4s$  and whenever $0<  \la^{\frac{2}{2-q}} + \delta^{\frac{2}{2-q}}< \Theta_0$ if $\mu>4s$}. Obviously they are distinct solutions because $ \mc N_{\la,\delta}^+ \cap \mc N_{\la,\delta}^- = \emptyset$. The proof is then completed using Proposition \ref{positive-sol}.\hfill{\QED}

\section*{Acknowledgments}The authors were funded by IFCAM (Indo-French Centre for Applied Mathematics) UMI CNRS 3494 under the project "Singular phenomena in reaction diffusion equations and in conservation laws"

 \linespread{0.5}

\end{document}